\documentclass[11pt]{amsart}

\usepackage{bbm}
\usepackage[margin=1in]{geometry}
\usepackage{amsmath,amssymb,amsthm,mathtools}
\usepackage{enumitem}
\usepackage[colorlinks=true,linkcolor=blue,citecolor=blue,urlcolor=blue]{hyperref}
\usepackage[nameinlink,capitalize]{cleveref}

\newtheorem{theorem}{Theorem}[section]
\newtheorem{conjecture}[theorem]{Conjecture}
\newtheorem{proposition}[theorem]{Proposition}
\newtheorem{lemma}[theorem]{Lemma}
\newtheorem{corollary}[theorem]{Corollary}
\newtheorem{notation}[theorem]{Notation}
\theoremstyle{definition}
\newtheorem{definition}[theorem]{Definition}

\newtheorem{remark}[theorem]{Remark}

\newtheorem{innercustomthm}{Theorem}
\newenvironment{customthm}[1]
  {\renewcommand\theinnercustomthm{#1}\innercustomthm}
  {\endinnercustomthm}

\newtheorem{innercustomprop}{Proposition}
\newenvironment{customprop}[1]
  {\renewcommand\theinnercustomprop{#1}\innercustomprop}
  {\endinnercustomprop}
\newcommand{\mult}{\text{mult}}
\newcommand{\length}{\text{length}}
\newcommand{\Fq}{\mathbb F_q}

\newcommand{\kk}{\Bbbk}

\newcommand{\A}{\mathbb{A}}
\newcommand{\al}{\alpha}
\newcommand{\B}{\mathcal{B}}
\newcommand{\PP}{\mathbb{P}}
\newcommand{\bG}{\bar{G}}
\newcommand{\PPc}{\check{\PP}}
\newcommand{\Wb}{\overline{W}}
\newcommand{\sg}{S_{\Gamma}}
\newcommand{\ag}{A_{\Gamma}}

\newcommand{\Fb}{\Bar{F}}

\newcommand{\TT}{\mathcal{T}}
\newcommand{\dqmk}
{(\underbrace{d,\ldots,d}_{m-k},
  \underbrace{q,\ldots,q}_{k})}
\newcommand{\fa}{f^{\A}}
\newcommand{\EGH}{\operatorname{EGH}}
\newcommand{\lc}{\check{l}}
\newcommand{\Pc}{\check{P}}
\newcommand{\Set}{\mathcal{S}}

\newcommand{\CC}{\mathcal{C}}

\newcommand{\fptm}{f^{\mathbb{P}}_{2}(d,m,m-2;q)}
\newcommand{\fA}{f^{\A}_r(d,m,k;q)}
\newcommand{\fP}{f^{\PP}_r(d,m,k;q)}
\newcommand{\eA}{e^{\mathbb{A}}_{r}(d,m,k;q)}
\newcommand{\eP}{e^{\mathbb{P}}_{r}(d,m,k;q)}

\newcommand{\dmkq}{(d,m,k;q)}
\newcommand{\fqb}{\overline{\mathbb{F}}_q}

\newcommand{\sat}{\text{sat}}
\newcommand{\lex}{\text{lex}}
\newcommand{\cdeg}{\text{c-}\deg}

\title{A $k$-Dimensional Version of the Largest Intersection Problem}
\author{Yuxin Lin}
\date{\today}

\begin{document}

\begin{abstract}
Suppose we have $r$ linearly independent hypersurfaces of degree $d$ in $\PP^m$ (or $\A^m$) defined over a finite field $\Fq$, whose intersection is at most $k$-dimensional. What is the largest possible number of $\Fq$-rational points in the intersection?
We conjecture an exact formula for this problem in both the projective and affine settings, assuming $q\geq d+1$.
The case $k=m-1$ recovers the Beelen-Datta-Ghorpade conjecture \cite{beelen2022combinatorial} and the case $k=0$ recovers the zero-dimensional conjecture in \cite{lin2026largest}.

Another interesting special case of the conjecture is $k=m-r$, which corresponds to the complete intersection of $r$ degree $d$ polynomials.
The case $r=1$, $k=m-1$ was proven by Serre in \cite{serre1991lettre} who showed that if $F$ is a degree $d$ homogeneous polynomial, then $|V(F)(\Fq)|\leq dq^{m-1}+\pi_{m-2}(q)$.
We prove the case $r=2$ and $k=m-2$, that is,
if $F_1, F_2$ are coprime, degree $d$ homogeneous polynomials, then $|V(F_1,F_2)(\Fq)|\leq d^2q^{m-2}+\pi_{m-3}(q)$.
\end{abstract}

\maketitle

\section{Introduction}

Let $\Fq$ be the finite field of size $q$.
Let $S(m,\Fq)=\Fq[X_0,\dots,X_m]$ and $S_d(m,\Fq)$ be its $d$-th graded component. Correspondingly, let $A(m,\Fq)=\Fq[x_1, \dots, x_m]$ and $A_{\leq d}(m,\Fq)$ be the vector subspace consisting of polynomials of degree $\leq d$. 
Given a subset $W\subseteq S_d(m,\Fq)$, we denote the common vanishing set by $V(W) \subseteq \PP^{m}$. Analogously for a subset $W\subseteq A_{\leq d}(m,\Fq)$, we denote the common zero set by $Z(W) \subseteq \A^{m}$.

The following problem has been studied in several recent papers. Given $r$ linearly independent polynomials in $S_d(m,\Fq)$ (respectively $A_{\leq d}(m,\Fq)$), what is the largest number of $\Fq$-rational points in the common vanishing set (respectively common zero set)?
For $1\leq r\leq \binom{m+d}{d}=\dim_{\Fq}(S_d(m,\Fq))$, these values are denoted by
\[
e^\PP_r(d,m;q)
:=\max\{|V(W)(\Fq)| : W\subseteq S_d(m,\Fq), \dim(W)=r\}.
\]
\[e^{\A}_r(d,m;q):=\max\{|Z(W)(\Fq)| : W\subseteq A_{\leq d}(m,\Fq), \dim(W)=r\}.\]
We would like to note that the problem of computing $e^{\PP}_r(d,m;q)$ (respectively $e^{\A}_r(d,m;q)$) is equivalent to computing the $r$-th generalized Hamming weight of the projective Reed-Muller code $\mathrm{PRM}_q(d,m)$ (respectively affine Reed-Muller code $\mathrm{RM}_q(d,m)$).

An exact formula for $e^{\A}_r(d,m;q)$ was proven by Heijnen and Pellikaan \cite{heijnen1998generalized} in 1998.
We will introduce some notations to state their result.
Let $\mathbb{N}$ be the set of non-negative integers. For positive integers $d$ and $m$, define
\[
\Omega(d,m)
:=\Big\{
(\beta_1,\dots,\beta_{m+1})\in \mathbb{N}^{m+1}
: \sum_{i=1}^{m+1} \beta_i=d
\Big\}.
\]
For $1\leq r\leq |\Omega(d,m)|=\binom{m+d}{d}$, let $w_r(d,m)=(\beta_1,\dots,\beta_{m+1})$ be its $r$-th largest element under lexicographical ordering. Then, define
\[
H_r(d,m;q)
:=\sum_{i=1}^m \beta_iq^{m-i}.
\]

\begin{theorem}\cite{heijnen1998generalized}\label{thm: affine BDG}
Given $m,d\geq 1$, $1\leq r\leq \binom{m+d}{d}$ and $q\geq d+1$, we have
\[
e^{\A}_r(d,m;q)
= H_r(d,m;q).
\]
\end{theorem}

Let $\pi_m(q):=|\PP^m(\Fq)|=\frac{q^{m+1}-1}{q-1}$. If $m<0$, we set $\pi_m(q)=0$.
Boguslavsky and Tsfasman \cite{boguslavsky1997number} in 1997 conjectured an exact formula for $e^\PP_r(d,m;q)$.


\begin{conjecture}[Boguslavsky-Tsfasman conjecture]\cite[Conjecture 3, Corollary 5]{boguslavsky1997number}\label{TBC}\\
Given $m,d\geq 1$, $1 \leq r \leq \binom{m+d}{d}$. Write $w_r(d,m)=(\beta_1,\dots,\beta_{m+1})$ and let $l$ be the smallest index such that $\beta_l\neq 0$. Then, for $q\geq d+1$, we have
$$e^{\PP}_r(d,m;q)=\sum_{i=l}^{m}\beta_{i}(\pi_{m-i}(q)-\pi_{m-i-l}(q)) +\pi_{m-2l}(q).$$
\end{conjecture}

It was shown by Datta and Ghorpade \cite{datta2015conjecture} that this conjecture is false for $r> m+1$.
Beelen, Datta and Ghorpade \cite{beelen2022combinatorial} conjectured a new exact formula for $e_r^{\PP}(d,m;q)$.

\begin{conjecture}[Beelen-Datta-Ghorpade Conjecture]\label{Conj: complete GDC}\cite{beelen2022combinatorial} \cite[Proposition 1.2]{lin2026largest}\\
Given $m,d\geq 1$, $1 \leq r \leq \binom{m+d}{d}$. Pick the unique $1\leq l\leq m+1$, for which
$$\tbinom{m+d}{d}-\tbinom{m+d+1-l}{d} <r \leq \tbinom{m+d}{d}-\tbinom{m+d-l}{d}.$$
Equivalently, write $w_r(d,m)=(\beta_1, \dots, \beta_{m+1})$, then $l$ is the smallest index for which $\beta_l \neq 0$.

Then for $q\geq d+1$ we have
\begin{equation}\label{Equation:projective}
e_r^\PP(d,m;q)=H_r(d,m;q)+\pi_{m-l-1}(q).
\end{equation}
\end{conjecture}

The formula of $e^{\PP}_r(d,m;q)$ has been proven in the following cases.
\begin{enumerate}
\item The case $r=1$ was shown by Serre \cite{serre1991lettre} and S{\o}rensen \cite{sorensen1991projective} both in 1991, although S{\o}rensen's proof is known to have a flaw.

\item The case $r=2$ was proven by Boguslavsky \cite{boguslavsky1997number} in 1997.

\item The case $d=2$ was proven by Zanella \cite{zanella1998linear} in 1998.

\item The case $r\leq m+1$, the case $d=1$ and the case $m=1$ were proven by Datta and Ghorpade \cite{datta2017number} in 2017.

\item The case $r\leq \binom{m+2}{2}$ was proven by Beelen, Datta and Ghorpade \cite{beelen2018maximum} in 2018.

\item The case $\binom{m+d}{d}-d\leq r\leq \binom{m+d}{d}$ was proven by Datta and Ghorpade \cite{datta2017remarks} in 2017. This range of $r$ corresponds to $l\in \{m,m+1\}$.
\item The case $m=2$ was proven by Lin and Singhal \cite{lin2026largest} in 2026.
\end{enumerate}
All of these cases involved specific ranges of $r$ or special values of $m,d$, but allowed for any $q\geq d+1$.
In \cite[Theorem 3]{singhal2025conjecture}, the authors showed that Conjecture~\ref{Conj: complete GDC} is true for any $m$, $d$ and $r$ when $q$ is sufficiently large.




\subsection{$k$-dimensional conjecture}\label{subsec: k-dim conj}

For $0 \leq k \leq m-1$ and $m-k\leq r\leq \binom{m+d}{d}$, denote
\begin{equation}
e_r^{\A}(d,m,k;q)
:=\max\{|Z(W)(\Fq)| : W\subseteq A_{\leq d}(m, \Fq), \dim(W)=r,\dim(Z(W))\leq k\}.
\end{equation}
\begin{equation}
e_r^{\PP}(d,m,k;q)
:=\max\{|V(W)(\Fq)| : W\subseteq S_{d}(m, \Fq), \dim(W)=r,\dim(V(W))\leq k\}.
\end{equation}

This definition is similar to $e_r^{\PP}(d,m;q)$ (respectively $e_r^{\A}(d,m;q)$), except for the requirement that $V(W)$ (respectively $Z(W)$) be of dimension at most $k$. We need $r\geq m-k$, as otherwise the vanishing set of $r$ polynomials cannot be of dimension~$\leq k$.

The major focus of this paper is to formulate  conjectures for exact formulae of $\eA$ and $\eP$.
We will introduce some notations to state our conjectures.

For integers $d,m \geq 1$, $0 \leq k \leq m-1$ and $q \geq d+1$, define
\begin{equation*}
\begin{split}
    \B(d,m,k;q):=\{(\alpha_1,\dots,\al_m) \in \mathbb{N}^{m} \mid  &\; 0 \leq \al_i \leq d-1 \text{ for } 1 \leq i \leq m-k,\\
    &\; 0 \leq \alpha_j \leq q-1 \text{ for } m-k+1 \leq j \leq m\}.
    \end{split}
\end{equation*}
For an element $w=(\al_1, \dots, \al_m) \in \B(d,m,k;q)$, let $\deg(w)=\sum_{j=1}^{m}\al_j$.

Define $\Omega'(d,m,k;q) \subseteq \B(d,m,k;q)$, such that
\begin{equation}\label{eq: omegap}
\begin{split}
    \Omega'(d,m,k;q):=\{w \in \B(d,m,k;q) \mid \deg(w) \leq d \}.
    \end{split}
\end{equation}

Notice that $\Omega'(d,m,k;q)=\Omega(d,m) \backslash \CC(d,m,k)$, where
\begin{equation}
    \CC(d,m,k):=\{(d,0,\dots,0), \dots, (\underbrace{0,\ldots,0,d}_{m-k},
  0,\ldots,0)\} .
\end{equation}
So $|\CC(d,m,k)|=m-k$ and hence
$|\Omega'(d,m,k;q)|=\binom{m+d}{d}-(m-k)$.

For $m-k\leq r\leq \binom{m+d}{d}$, let $w'_{r-(m-k)}(d,m,k;q)=(\al_1, \dots, \al_m)$ be its $(r-(m-k))$-th largest element under lexicographical ordering. Then, define:
\begin{equation}\label{eq: formula of frA}
\fA :=\sum_{j=1}^{m-k-1}\al_j d^{m-k-j}q^k +\sum_{j=m-k}^m \al_j q^{m-j}.
\end{equation}
If $r=m-k$, then $(d,0, \dots,0)$ is considered to be the $0$-th largest element of $\Omega'(d,m,k;q)$ and
\[
f^{\A}_{m-k}(d,m,k;q) :=d^{m-k} q^k .
\]

We conjecture the following.

\begin{conjecture}[Affine conjecture]\label{conj:affine conj 1}
Given $d,m\geq 1$, $0\leq k\leq m-1$, $m-k\leq r\leq \binom{m+d}{d}$ and $q\geq d+1$,
we have
\[
  \eA = \fA.
\]
\end{conjecture}

\begin{conjecture}[Projective conjecture]\label{conj:projective-geometric}
Given $d,m\geq 1$, $0\leq k\leq m-1$, $m-k\leq r\leq \binom{m+d}{d}$ and $q\geq d+1$, write $w'_{r-(m-k)}(d,m,k;q)=(\al_1, \dots, \al_m)$. Let $l$ be the smallest index such that $\al_l \neq 0$.
Let $s = \min(m-l-1,k-1)$. 
Then we have
\[
\eP = \fA+\pi_{s}(q).
\] 
\end{conjecture}

We denote this conjectured formula by
\[
\fP:=\fA+\pi_{s}(q).
\]
When $r=\binom{m+d}{d}$, then $w'_{r-(m-k)}(d,m,k;q)=(0, \dots, 0)$ and we take $l=m+1$.

Note that our conjectures provide a natural formalism that connects Conjecture \ref{Conj: complete GDC} and \cite[Conjecture 1.6]{lin2026largest}.
\begin{enumerate}
\item When $k=m-1$, we have $\fA=H_r(d,m;q)$  and $\fP=H_r(d,m;q)+\pi_{m-l-1}(q)$. Therefore, the statement of Conjecture~\ref{conj:affine conj 1} becomes Theorem~\ref{thm: affine BDG}, and the statement of Conjecture~\ref{conj:projective-geometric} becomes Conjecture~\ref{Conj: complete GDC}.

\item When $k=0$, we have $\fA=\fP$ and this is the same as the conjectured formula of \cite[Conjecture 1.6]{lin2026largest} (denoted there as $H'_{r-(m-1)}(d,m)$). Therefore, for $k=0$, the statements of Conjecture~\ref{conj:affine conj 1} and Conjecture~\ref{conj:projective-geometric} become \cite[Conjecture 1.6]{lin2026largest}.
    
\item When $r=m-k$, then Conjecture \ref{conj:affine conj 1} states that the complete intersection of $m-k$ degree $d$ hypersurfaces in $\A^m$ has at most $d^{m-k}q^k$ rational points over $\Fq$. This is proven by Lachaud and Rolland in \cite{lachaud2015number}.

\item  When $r=m-k$, then Conjecture \ref{conj:projective-geometric} states that the complete intersection of $m-k$ degree $d$ hypersurfaces in $\PP^m$ has at most $d^{m-k}q^k+\pi_{k-1}(q)$ rational points over $\Fq$. Lachaud and Rolland construct examples of complete intersections that achieve this bound in \cite{lachaud2015number}, but the bound is only proven for $r=1$, $k=m-1$.
\end{enumerate}

In Proposition \ref{prop :affine-geometric-equality} and Proposition \ref{prop: equal proj}, we show that our conjectured formulae are at least lower bounds for $e_r^{\A}(d,m,k;q)$ and $e_r^{\PP}(d,m,k;q)$.
\begin{proposition}
For positive integers $m,d,r$ with $0 \leq k \leq m-1$, $m-k\leq r\leq \binom{m+d}{d}$ and $q\geq d+1$, we have
\[
e_r^{\A}(d,m,k;q)\geq \fA,
\]
\[\eP \geq \fP.\]
\end{proposition}

We show that the projective conjecture for $k=m-2$ implies Conjecture \ref{Conj: complete GDC}.
\begin{proposition}\label{prop: projective m-2 implies BDG} 
Suppose $d\geq 1, m \geq 2,2\leq r\leq \binom{m+d}{d}$ and $q\geq d+1$. 
If we know that for each $1\leq d'\leq d$ satisfying $r\leq \binom{m+d'}{d'}$,
\[
e^{\PP}_r(d',m,m-2;q) 
=f^{\PP}_r(d',m,m-2;q),
\]
then we have that 
\[
e^{\PP}_r(d,m,m-1;q) 
=f^{\PP}_r(d,m,m-1;q).
\]
\end{proposition}

Finally, we focus on Conjecture \ref{conj:affine conj 1} and Conjecture \ref{conj:projective-geometric} in the special case of complete intersections, that is $r=m-k$. 
For $r=m-k$, the conclusion of Conjecture \ref{conj:affine conj 1} was proven by \cite{lachaud2015number}. 
In this paper, we prove Conjecture \ref{conj:projective-geometric} in the case where $k=m-2$ and $r=2$.

\begin{theorem}\label{prop: r=m-k,k=m-2}
Given positive integers $d,m$ with $m \geq 2$ and $q\geq d+1$, we have
\[
e^{\PP}_2(d,m,m-2;q)=f_2^{\PP}(d,m,m-2;q).
\]
\end{theorem}


We also state a related conjecture about standard graded Artinian algebras, and we will later discuss why this conjecture motivates the previous ones.
We will also show that the Eisenbud-Green-Harris Conjecture (see Section~\ref{sec: prelim}) implies this conjecture.

\begin{conjecture}[Affine Algebraic Conjecture]\label{conj: affine conjecture 2} 
Let $\kk=\fqb$. Let $A= \kk[x_1,\ldots,x_m]/I$ be a standard graded Artinian algebra, where $I$ contains a regular sequence of degree $(\underbrace{d,\ldots,d}_{m-k\text{ times}}, \underbrace{q,\ldots,q}_{k\text{ times}})$. Write $A_{\le d} := \bigoplus_{i=0}^d A_i$ to be the degree $\leq d$ graded part of~$A$. Then
\[
\max\{ \dim_{\kk} A \mid   
\dim_{\kk} A_{\le d}
\le
\tbinom{m+d}{d}-r \}
=\fA.
\]
\end{conjecture}


This paper is organized as follows. We introduce some preliminaries on the Artinian reduction and Eisenbud-Green-Harris Conjecture in Section~\ref{sec: prelim}. 
In Section~\ref{sec: EGH and affine}, we discuss Conjecture \ref{conj: affine conjecture 2} and prove that it is implied by the EGH conjecture. 
In Section~\ref{sec: affine conjecture equality}, we discuss Conjecture~\ref{conj:affine conj 1} and show that the conjectured formula for $\eA$ is at least a lower bound. 
In Section~\ref{sec: projective conjecture}, we discuss Conjecture~\ref{conj:projective-geometric} and show that the conjectured formula for $\eP$ is at least a lower bound. 
In Section \ref{sec: m-2 proj implies B-D-G}, we show that Conjecture~\ref{conj:projective-geometric} at dimension $k=m-2$ implies Conjecture~\ref{Conj: complete GDC}. In Section~\ref{sec: codim 2}, we focus on the case of complete intersection and show Theorem \ref{prop: r=m-k,k=m-2}.

\section{Preliminaries}\label{sec: prelim}
\subsection{Artinian reduction $A_{\Gamma}$}
In this section, we recall the construction of the Artinian reduction $A_{\Gamma}$ from a zero-dimensional scheme $\Gamma \subseteq \PP^m$. The construction was given in detail by Eisenbud, Green and Harris \cite{eisenbud1996cayley}.

Throughout this paper, let $\kk=\fqb$ and let $S(m,\kk):=\kk[X_0, \dots, X_m]$ be the homogeneous coordinate ring of $\PP^m$, and $A(m,\kk):=\kk[x_1, \dots, x_m]$ be the coordinate ring of $\mathbb{A}^m$.

Let $\Gamma \subseteq \PP^m$ be a zero-dimensional scheme.  Let $I(\Gamma) \subseteq S(m,\kk)$ be the homogeneous ideal of $\Gamma$, and let $S_{\Gamma}:=S(m,\kk)/I(\Gamma)$ be the homogeneous coordinate ring of $\Gamma$. 
Let $\PPc^m(\kk)$ be the set of hyperplanes of $\PP^m$.
Since $\kk$ is infinite and $\operatorname{Supp}(\Gamma)$ is finite, there exists a hyperplane $ H \in \PPc^m(\kk)$ disjoint from $\Gamma$.

Now, let $H \in \PPc^m(\kk)$ be a hyperplane such that $H \cap \Gamma=\emptyset$. Let $h \in S_1(m,\kk)$ be the linear form defining this hyperplane. Since $\Gamma$ is zero-dimensional, $S_{\Gamma}$ is a one-dimensional Cohen-Macaulay ring. Since $H\cap\Gamma=\emptyset$, $h$ is not a zero-divisor in $S_{\Gamma}$. The Artinian reduction associated to $\Gamma$ is:
\[
A_{\Gamma}:=S_\Gamma/(h)=S(m,\kk)/(I(\Gamma),h).
\]
By its construction, $A_{\Gamma}$ has Krull dimension $0$, hence it is a local Artinian algebra with the unique maximal ideal $\mathfrak{m}=(X_0, \dots, X_m)$. For a graded algebra $A$, let $A_i$ be its $i$-th graded piece.
We have the following exact sequence on each graded piece:
\[
0 \to {(\sg)}_{i-1} \xrightarrow{\cdot h}{(\sg)}_i \to {(\ag)}_{i} \to 0.
\]
Let $h_{\ag}$ (resp. $h_{\sg}$) be the Hilbert function of the corresponding graded algebra. From the above exact sequence, we have
\[
h_{\ag}(i)=h_{\sg}(i)-h_{\sg}(i-1).
\]
Taking the sum over all graded pieces, we get the following two properties of $\ag$:
\begin{enumerate}
    \item $\dim_\kk(\ag)
    =\lim_{N \to \infty}h_{\sg}(N)=\deg(\Gamma)$,
    \item $\dim_{\kk}({\ag}_{\leq d})
    =h_{\sg}(d)=\binom{m+d}{d}- \dim_{\kk}( I(\Gamma)_d)$.
\end{enumerate}
\subsection{Affine $k$-dimensional conjecture}
We apply the Artinian reduction in our setting.

Let $W \subseteq S_d(m, \Fq)$ be a vector space such that $\dim_{\Fq}(W)=r$ and $\dim(V(W)) \leq k$. Let $I(W)$ be the ideal generated by $W$ in $S(m,\kk)$. Let $W_{\kk}:=W \otimes_{\Fq}\kk$.

Consider the zero-dimensional subscheme $\Gamma =V(W)(\Fq)$. By \cite[6.1]{beelen2022combinatorial}, $I(\Gamma)=(I(W), G_{ij})^{\sat}$ where $G_{ij}=X_iX_j^{q}-X_jX_i^{q}$ for $ 0 \leq i<j\leq m$. Notice that the common vanishing locus of $\{G_{ij}\}_{0\leq i<j \leq m}$ consists of exactly the $\Fq$-rational points in $\PP^m$.

Now, pick a hyperplane $H \in \PPc^m(\kk)$ such that $H\cap\Gamma=\emptyset$ and let $h \in S_1(m,\kk)$ be its defining linear form. From the above construction we get the Artinian reduction $\ag=S(m,\kk)/(h,I(\Gamma))$. It has the following properties:
\begin{itemize}
    \item $\dim_{\kk}(\ag)=|V(W)(\Fq)|$,
    \item $\dim_{\kk}({\ag}_{\leq d}) \leq \binom{m+d}{d}-r$.
\end{itemize}

We have the following lemma saying we can extract a regular sequence from the ideal of $\ag$.
\begin{lemma}\label{lem: regular sequence Krull dimension}
Let $A$ be a graded Cohen-Macaulay $\kk$-algebra of Krull dimension $m$. Let $W \subseteq A$ be a finite dimensional vector subspace over $\kk$ and $I(W)$ be the ideal generated by $W$. Let $B=A/I(W)$. If the Krull dimension of $B$ is at most $m-k$, then we can find a regular sequence $(F_1,F_2, \dots, F_{k})$ in~$W$.
\end{lemma}
\begin{proof}
    We show that we can find a regular sequence $(F_1, \dots, F_{k'})$ for $0 \leq k' \leq k$. We induct on $k'$. The base case $k'=0$ is vacuously true. 

    Now, suppose we have a regular sequence $(F_1, \dots, F_{k'-1})$. Then $A/(F_1, \dots, F_{k'-1})$ is Cohen-Macaulay with Krull dimension $m+1-k' \geq m+1-k$. 

    Let $\{I(\Gamma_i)\}$ be the minimal primes in $A$ containing the ideal $(F_1, \dots, F_{k'-1})$. Since the quotient is Cohen-Macaulay, the minimal primes are exactly the associated primes of the quotient, and an element $F$ in the quotient ring is not a zero-divisor if and only if it avoids all the $I(\Gamma_i)$. By Krull's height theorem, each $I(\Gamma_i)$ has height $k'-1 \leq k-1$. Hence, for each $I(\Gamma_i)$, there exists $H_i \in W$ such that $H_i \not \in I(\Gamma_i)$. Indeed, if every $F\in W$ belongs to $I(\Gamma_i)$, then
$I(W)\subseteq I(\Gamma_i)$. Hence
$\dim(A/I(W)) \geq \dim(A/I(\Gamma_i))=m-k'+1>m-k$,
contradicting the assumption that $\dim(B)\leq m-k$.

Now, for each $I(\Gamma_i)$ we choose $H_i \not \in I(\Gamma_i)$. By the prime-avoidance lemma, we can find some $\kk$-linear combination $F=\sum_i a_iH_i$ such that $F \not \in I(\Gamma_i)$ for any $i$.  Hence, we extend the regular sequence $(F_1, \dots, F_{k'-1})$ to $( F_1, \dots, F_{k'})$, which completes the induction step.
\end{proof}
We want to show that there is an Artinian reduction $A_{\Gamma}$ that contains a regular sequence of degree $(\underbrace{d, \dots, d}_{m-k}, \underbrace{q+1, \dots, q+1}_{k})$ in its defining ideal. Applying Lemma
\ref{lem: regular sequence Krull dimension} to $S(m,\kk)$ and the vector space $W_{\kk}$, we can find a regular sequence $(F_1,\dots,F_{m-k})$ where $F_i\in W_{\kk}$. Applying Lemma \ref{lem: regular sequence Krull dimension} to $S(m,\kk)/(F_1, \dots, F_{m-k})$ and the vector space $\langle G_{ij}\rangle$,  we can extend the regular sequence to
$(F_1,\dots,F_{m-k},G_1,\dots,G_k)$,
where $G_j\in\langle G_{ij}\rangle$. Since $V(F_1,\dots,F_{m-k},G_1,\dots,G_k)$ is zero-dimensional, we can
find a hyperplane $H=V(h)\in\PPc^m(\kk)$ that avoids its support.
The Artinian reduction
$A_{\Gamma}=S(m,\kk)/((I(W),G_{ij})^{\sat},h)$ is then an Artinian algebra
whose defining ideal contains a regular sequence of degree
$(\underbrace{d,\dots,d}_{m-k},
\underbrace{q+1,\dots,q+1}_{k})$.

Now suppose there is a hyperplane $H$ such that $H\cap \Gamma=\emptyset$, $H$ is defined over $\Fq$ and $\dim(H\cap V(W))\leq k-1$.
After an $\Fq$-rational change of variables, we can assume that $H=V(X_0)$.
Let $U(h)$ be the non-vanishing locus of a linear form $h$. Then, the $\Fq$-rational points in $U(X_0)$ are the common vanishing locus of $X_i^q-X_iX_0^{q-1}$ for $1 \leq i \leq m$.
We know that $\Gamma=V(W)(\Fq)\subseteq U(X_0) \cong \A^m$. Then $I(\Gamma)=(I(W),G_1, \dots, G_m)^{\sat}$ where $G_i=X_i^q-X_iX_0^{q-1}$. 
Denote $\overline{I(\Gamma)} = (I(\Gamma)+(X_0))/(X_0)$.
Hence, we can take 
\[
\ag
=\kk[X_0, \dots, X_m]/(X_0, I(\Gamma))
\cong \kk[X_1,\dots, X_m]/ \overline{I(\Gamma)}.
\] 
Since $\dim(H\cap V(W))\leq k-1$, by Lemma \ref{lem: regular sequence Krull dimension}, we can find a regular sequence of length $m-k$ in $(W_{\kk}+(X_0))/X_0 \subseteq (\overline{I(\Gamma)})_d$.
Moreover, $\bG_i=X_i^q$, so $\kk[X_1,\dots,X_m]/(\bG_1,\dots,\bG_m)$ has Krull dimension $0$. 
Thus, applying Lemma \ref{lem: regular sequence Krull dimension} to the quotient by the previous regular sequence and the vector space spanned by the images of $\bG_1,\dots,\bG_m$, we can extend the previous regular sequence to one of degree $(\underbrace{d, \dots, d}_{m-k},\underbrace{q, \dots, q}_k)$ in $\overline{I(\Gamma)}$.
Moreover, we have $\dim_{\kk}({\ag}_{\leq d})\leq \binom{m+d}{d}-r$ and $\dim_{\kk}(\ag)=|V(W)(\Fq)|$. 
Conjecture~\ref{conj: affine conjecture 2} is thus an algebraic version of Conjecture \ref{conj:affine conj 1}.


Even though Conjecture \ref{conj: affine conjecture 2} does not imply Conjecture \ref{conj:affine conj 1} or Conjecture \ref{conj:projective-geometric}, it relates to these two conjectures under some additional conditions.
\begin{enumerate}
    \item Relation to Conjecture \ref{conj:affine conj 1}. Start with $W \subseteq A_{\leq d}(m, \Fq)$ with $\dim_{\Fq}(W)=r$ and $\dim(Z(W)) \leq k$, let $\Gamma=Z(W)(\Fq) \subseteq \A^m(\Fq)$. Let $U(X_0)$ be the non-vanishing locus of $X_0$. We may identify $\A^m$ with $U(X_0) \subseteq \PP^m$ and realize $\Gamma$ as a zero-dimensional scheme in $\PP^m$. Then $H=V(X_0)$ does not intersect $\Gamma$.  Constructing the Artinian reduction $\ag$ as above, we have $\dim_\kk({\ag}_{\leq d}) \leq \binom{m+d}{d}-r$. If, in addition, the ideal of $\ag$ contains a regular sequence of degree $(\underbrace{d, \dots, d}_{m-k},\underbrace{q, \dots, q}_k)$, then Conjecture \ref{conj: affine conjecture 2} implies that $|Z(W)(\Fq)| \leq \fA$, which is the conclusion of Conjecture \ref{conj:affine conj 1}. However, unlike the projective case, $\dim(Z(W)) \leq k$ does not imply the existence of a regular sequence of the required degree. 
    \item Relation to Conjecture \ref{conj:projective-geometric}. Start with $W \subseteq S_{d}(m, \Fq)$ with $\dim_{\Fq}(W)=r$ and $\dim(V(W)) \leq k$.
    \begin{itemize}
    \item When $k=0$, then Conjecture \ref{conj: affine conjecture 2} implies Conjecture \ref{conj:projective-geometric}. Indeed, take $\Gamma=V(W)$, then $\Gamma$ is already a zero-dimensional scheme. The ideal of the Artinian reduction $\ag$ contains a regular sequence of degree $(\underbrace{d, \dots, d}_{m})$, and the conclusion of Conjecture \ref{conj: affine conjecture 2} implies that \[
    |V(W)(\Fq)| \leq \deg(V(W)) \leq f^{\A}_r(d,m,0;q)=f^{\PP}_r(d,m,0;q).
    \]
    \item For $1 \leq k \leq m-1$, take $\Gamma=V(W)(\Fq)$. Let $s=\min(m-l-1,k-1)$ where $l$ is the smallest index that $\al_l \neq 0$ for $w'_{r-(m-k)}(d,m,k;q)=(\al_1, \dots, \al_m)$. If there exists a $\Fq$-rational hyperplane $H=V(h)$ such that $|H \cap \Gamma| \leq \pi_s(q)$, then both Conjecture~\ref{conj: affine conjecture 2} and Conjecture~\ref{conj:affine conj 1} independently imply 
    \[|V(W)(\Fq)| \leq  f^{\PP}_r(d,m,k;q) =f^{\A}_r(d,m,k;q)+\pi_{s}(q).\]   
    
    First, assume Conjecture~\ref{conj: affine conjecture 2}.
    Let $\Gamma'=\Gamma \cap U(h)$. By linear change of variables, we can assume that $h=X_0$.  By Lemma~\ref{lem: regular sequence Krull dimension}, there is a 
    regular sequence $( F_1,\dots, F_{m-k}) \subseteq W_{\kk}$.
    Let $A_1=S(m,\kk)/\langle F_1,\dots,F_{m-k}\rangle$. Then $A_1$ is Cohen-Macaulay of Krull dimension $k+1$. 
    Let $W_1$ be the vector subspace of $A_1$ generated by the images of $X_i^q -X_i X_0^{q-1}$ for $1\leq i\leq m$. By Lemma~\ref{lem: regular sequence Krull dimension} there is a regular sequence $( G_1,\dots, G_k)$ in $W_1$. Thus, $( F_1, \dots, F_{m-k}, G_1, \dots, G_k)$ is a regular sequence in $I(\Gamma')\subseteq \kk[X_0,\dots, X_m]$ of degree $(\underbrace{d, \dots, d}_{m-k}, \underbrace{q, \dots, q}_{k})$.
    Pick a hyperplane $V(h')$ that does not intersect the zero-dimensional scheme $V(F_1, \dots, F_{m-k}, G_1, \dots, G_k)$. This means that $( F_1, \dots, F_{m-k},\\ G_1, \dots, G_k, h')$ is a regular sequence in $\langle I(\Gamma'),h'\rangle$. Let $A_{\Gamma'}=S(m,\kk)/(I(\Gamma'),h')$. By Conjecture~\ref{conj: affine conjecture 2},
    \[
    \deg(\Gamma') =\dim(A_{\Gamma'})
 \leq f^{\A}_r(d,m,k;q).
    \]
    Hence
    \[|V(W)(\Fq)| \leq  f^{\A}_r(d,m,k;q)+\pi_{s}(q).\]

    Next assume Conjecture \ref{conj:affine conj 1}. Let
$F_1,\dots,F_r$ be an $\Fq$-basis of $W$.
    If we let the dehomogenization of $(F_1, \dots, F_r)$ with respect to $h$ be $(f_1, \dots, f_r)$, then 
    \[Z(f_1, \dots, f_r)=V(F_1, \dots, F_r) \cap U(h).\] 
    By Conjecture \ref{conj:affine conj 1}, we have $|Z(f_1, \dots, f_r)(\Fq)| \leq \fA$, so 
    \[
    |V(F_1, \dots, F_r)(\Fq)| \leq \fA+\pi_s(q)=\fP.
    \] 
    \end{itemize}
\end{enumerate}

\subsection{Lex-plus-powers monomial ideals and the EGH conjecture}\label{Sec: EGH}

In this section, we let $S:=S(m-1,\kk)=\kk[x_1,\dots,x_m]$ be the standard graded polynomial ring in $m$ variables. Equip $S$ with the lexicographic order on monomials induced by $x_1>x_2>\cdots>x_m$.
\begin{definition}[Monomial ideal, {\cite[Definition 2.1.4]{coxLittleOSheaIVA}}]
An ideal \(I\subseteq S\) is called a
\textbf{monomial ideal} if it is generated by monomials. That is, there
exists a subset \(A\subseteq \mathbb N^m\) such that
\[
I=\langle x^\alpha:\alpha\in A\rangle,
\]
where $x^\alpha=x_1^{\alpha_1}\cdots x_m^{\alpha_m}$.
\end{definition}
From the definition, we see that for a monomial ideal $I$, if a polynomial $F \in I$, then each of the monomials appearing in $F$ must also belong to $I$. Equivalently, if $B=S/I$ is the quotient of $S$ by a monomial ideal $I$, then the non-zero monomials in $B$ form a $\kk$-basis of $B$.

\begin{definition}[Lex-ideal, {\cite[Definition 2.1]{gunturkun2021survey}}]
    A monomial ideal $J \subseteq S$ is called a \textbf{lex ideal} if, for every monomial $M \in J$, for every other monomial $N$ such that $\deg(N)=\deg(M)$ and $N>_{\lex}M$, we have $N \in J$. 
\end{definition}
\begin{definition}[Lex-plus-powers ideal, {\cite[Definition 2.1]{gunturkun2021survey}}]\label{def: lex-plus-power ideal}
A monomial ideal \(L\subseteq S\) is called a \textbf{lex-plus-powers ideal} of degree
\(\mathbf a=(a_1\leq\dots\leq a_m)\) if
\[
L=(x_1^{a_1},\dots,x_m^{a_m})+J,
\]
where \(J\subseteq S\) is a lex ideal.
\end{definition}
Now we summarize some properties of the quotient of $S$ by a lex-plus-powers ideal $L$. These are standard consequences of the definitions.
\begin{proposition}[Properties of lex-plus-powers model]\label{prop: properties of lex-plus-power}
Let $L$ be a lex-plus-powers ideal of degree $\mathbf a=(a_1,\dots,a_m)$. Let $B=S/L$. Let $B_i$ be the subspace of degree $i$ polynomials in $B$.
\begin{enumerate}
    \item Let $M$ be a non-zero monomial in $B$. If $M'$ is a monomial such that $M' \mid M$, then $M'$ is non-zero in $B$.
    \item Let $h=\dim_{\kk}B_i$. Then, $B_i$ has a basis consisting of the $h$ lex-smallest monomials in $\B(\mathbf{a})_i$, where 
    \[
    \B(\mathbf{a})_i:=\{x_1^{\al_1} \dots x_{m}^{\al_m} \mid 0 \leq\al_j \leq a_j-1, \sum_{j=1}^m\al_j=i\}.
    \]
\end{enumerate}
    
\end{proposition}





\begin{proof}
\begin{enumerate}
    \item Since \(M'\mid M\), we can write $M=M' N$ . If \(M'\in L\), then we have $M \in L$, 
which contradicts the assumption that \(M\) is nonzero in \(B=S/L\).
\item 
Since \(L\) is a monomial ideal, \(B_i\) has basis given by the degree \(i\)
monomials not contained in \(L\). Since
$(x_1^{a_1},\dots,x_m^{a_m})\subseteq L$,
every surviving monomial lies in the box
\[
\B(\mathbf{a}):=\{x_1^{\al_1} \dots x_{m}^{\al_m} \mid 0 \leq\al_j \leq a_j-1\}.\]

By Definition \ref{def: lex-plus-power ideal}, if \(x^\alpha\in L\) and
\(x^\beta>_{\lex}x^\alpha\), with \(\alpha,\beta\in \B(\mathbf a)_i\), then
\(x^\beta\in L\), hence zero in $B$. Therefore the non-zero monomials in $B_i$ are the lex-smallest monomials. Since the number of non-zero monomials in $B_i$ is $h$, they are exactly the \(h\) lex-smallest monomials in \(\B(\mathbf a)_i\).
\end{enumerate}
\end{proof}
As a consequence, for a lex-plus-powers quotient $B=S/L$, given $h=\dim_{\kk}B_i$, one
can obtain lower bounds for the Hilbert function in smaller degrees. Namely,
for \(j<i\), every degree \(j\) divisor of a standard monomial in \(B_i\) is
again nonzero in \(B_j\). Since distinct surviving monomials are linearly
independent modulo a monomial ideal, these divisors give a linearly independent
subset of \(B_j\). Therefore
\[
h_B(j)
\ge
\#\left\{
\text{degree } j \text{ monomial divisors of the lex smallest $h$ monomials in } \B(\mathbf a)_i
\right\}.
\]

We now recall the Eisenbud-Green-Harris conjecture in the form needed below.

\begin{conjecture}[Eisenbud-Green-Harris]\cite[Conjecture 3.4]{gunturkun2021survey}\label{conj:EGH} \\
Suppose $I \subseteq S$ is a homogeneous ideal containing a regular sequence $(f_1, \dots, f_m)$ with $\deg(f_i)=a_i$ and $2 \leq a_1\leq \dots \leq a_m$. 
Then there exists a lex-plus-powers ideal $
L=(x_1^{a_1},\dots,x_m^{a_m})+J$
such that
\[
h_{S/I}(t)=h_{S/L}(t)
\qquad\text{for all }t\ge 0.
\]
\end{conjecture}
When such an \(L\) exists, we call \(S/L\) the \textbf{lex-plus-powers model} of $A=S/I$.

In this paper, the degree sequence of interest is
$\mathbf a=(\underbrace{d,\dots,d}_{m-k\text{ times}},
\underbrace{q,\dots,q}_{k\text{ times}})$.
We write this sequence as $(d^{m-k},q^k)$. The corresponding box of monomials is:
\begin{equation*}
\begin{split}
\B
=\B(d,m,k;q)
:=\{(\alpha_1,\dots,\al_m) \in \mathbb{N}^{m} \mid\;  & 0 \leq \al_i \leq d-1 \text{ for } 1 \leq i \leq m-k,\\
    & 0 \leq \alpha_j \leq q-1 \text{ for } m-k+1 \leq j \leq m\}.
    \end{split}
\end{equation*}
For an element $w=(\al_1, \dots, \al_m) \in \B$, let $\deg(w)=\sum_{j=1}^{m}\al_j$.
\begin{definition}[Ordering]
We define the following orderings in $\B$.
For $w, w' \in \B$, we say that:
\begin{enumerate}
    \item $w$ is \textbf{lexicographically smaller} than $w'$ (written as $w< _{\lex}w'$) if in the first index $p$ where $w_p \neq w'_p$, we have $w_p < w'_p$.
    \item $w$ is \textbf{coordinately smaller} than $w'$ (written as $w\prec w'$) if $w_i \leq w'_i$ for $1 \leq i \leq m$, and $w \neq w'$. 
\end{enumerate}
\end{definition}

Notice that $<_{\lex}$ is a total ordering on $\B$, whereas $\prec$ is a partial ordering.

\begin{definition}
Let $A$ be a standard graded Artinian algebra that is a quotient of $S=\kk[x_1, \dots, x_m]$, whose ideal contains a regular sequence of degree $(d^{m-k}, q^k)$.
We say that $\text{EGH}(d^{m-k}, q^k)_m$ holds for $A$, if $A$ has a lex-plus-powers model of degree $(d^{m-k}, q^k)$. That is, if there exists a lex-plus-powers ideal $L$ of degree $(d^{m-k}, q^k)$ where $h_{S/L}(i)=h_A(i)$ for all $i \geq 0$. 
\end{definition}

\begin{remark}
    By \cite[Conjecture 3.4]{gunturkun2021survey}, it is expected that if $h_A(1)=m$ and the ideal $I$ of $A$ contains a regular sequence of degrees $\dqmk$, then $A$ satisfies $\EGH(d^{m-k},q^k)_m$. The $\EGH$ conjecture as stated in \cite[Conjecture 3.4]{gunturkun2021survey} is known to be true in the following cases:
    \begin{itemize}
        \item $m=2$.
        \item $m=3$ and the degree of the regular sequence is $(2,d,d)$ or $(3,d,d)$.
        \item $m=3$ and $A$ is Gorenstein.
    \end{itemize}
\end{remark}

If $A$ satisfies $\EGH(d^{m-k},q^k)_m$, let $S/L$ be a lex-plus-powers
model of $A$. In particular we know that the monomials in $S/L$ form a $\kk$ basis. We define
$$
\Lambda(A):=
\{(\al_1,\dots,\al_m)\in\B(d,m,k;q)
\mid x_1^{\al_1}\cdots x_m^{\al_m}\notin L\}.
$$


\begin{definition}[closed and compressed]
Let $\Lambda \subseteq \B$. Write $\Lambda(i)$ (resp. $\B(i)$) for the elements in $\Lambda$ (resp. $\B$) of degree $i$.
\begin{enumerate}
    \item  We say that $\Lambda$ is \textbf{closed} if $\forall w \in \Lambda$, $\forall w' \preccurlyeq w$, we have $w' \in \Lambda$.
    \item We say that $\Lambda$ is \textbf{compressed} if for each degree $i$, $\Lambda(i)$ consists of the lex-minimal $|\Lambda(i)|$ elements in $\B(i)$. That is,  $ \forall w \in \Lambda(i)$,  $ \forall w' \in \B(i)$ and $w'\leq_{\lex}w$, we have $w' \in \Lambda$.
\end{enumerate} 
\end{definition}

By Proposition \ref{prop: properties of lex-plus-power}, we see that if $A$ satisfies $\EGH(d^{m-k},q^k)_m$, then $\Lambda(A)$ is closed and compressed.  

\subsection{The conjectured formula $\fA$}\label{sec: conjecture formula fA}
In this section, we give a natural interpretation for the formula $\fA$ defined in \eqref{eq: formula of frA} and show that it is a non-decreasing function with respect to $k$ when other parameters are fixed. 
We fix the following parameters:
    \begin{enumerate}
        \item $m,d \geq 1$, and the prime power $q \geq d+1$.
        \item $1 \leq r \leq \binom{m+d}{d}$.
    \end{enumerate}
Let $0 \leq k \leq m-1$.
For the sake of simplicity, we drop the index $r,m,d,q$ from the notations and write them as a function of $k$. 
\begin{notation}\label{notation: function of k}
We will use the following notations in this subsection:
\begin{enumerate}
    \item $s(k):=r-m+k$,
    \item $\Omega'(k):=\Omega'\dmkq$,
    \item $\B(k):=\B\dmkq$,
    \item $\CC(k):=\CC(d,m,k)$,
    \item $\fa(k)=\fA$,
    \item If $k \geq m-r$: $w'_{s(k)}=w'_{r-(m-k)}\dmkq$.
\end{enumerate}
\end{notation}

For $k$ such that $s(k)\geq 1$, we define the subset
$\Lambda_{s(k)}\subseteq \B(k)$ by
\begin{equation}\label{eq:lambdas}
\begin{split}
\Lambda_{s(k)}
=
\{w\in\B(k)\mid w<_{\lex}w'_{s(k)}\}.
\end{split}
\end{equation}
We adopt the convention that if $s(k) \leq 0$, equivalently when $k \leq m-r$, then we have $\Lambda_{s(k)}=\B(k)$.

Notice that $\fA$ in \eqref{eq: formula of frA} has the following natural interpretation
\begin{equation}
\fa(k)
= |\Lambda_{s(k)}|
=|\{w \in \B(k) \mid w <_{\lex}w'_{s(k)}\}|.
\end{equation}
Following the above convention, when $s(k)=r-(m-k)\leq 0$, we have 
$\fa(k)=|\B(k)|=d^{m-k}q^k$. Whenever $w'_{s(k)}$ is used below, we implicitly assume that
$s(k)\geq 0$, equivalently $k\geq m-r$.
Given the parameters $m,k$ and $r< \binom{m+d}{d}$, we are interested in the smallest non-zero index of $w'_{s(k)}$.
We make the following definition to characterize the range of $r$, for which this index is $l$.

\begin{definition}\label{def: Fmdk}
For $0 \leq l \leq m$, we define the expression $F_{k}(l)$ such that:
\begin{equation*}
F_{k}(l) =
   \begin{cases}
    -1 & \text{ if $l=0$}\\
   \binom{m+d}{d}-\binom{m+d-l}{d}-\min(m-k,l) & \text{ for $l \geq 1$}
\end{cases} 
\end{equation*}
\end{definition}
\begin{lemma}\label{Lem: range for l}
    Let $l$ be the unique number such that $F_{k}(l-1)<s(k) \leq F_{k}(l)$. Write \\
    $w'_{s(k)}=(\al_1,\dots, \al_m)$, then $l$ is the smallest index such that $\al_l \neq 0$.
\end{lemma}

\begin{proof}
Let $e_l=(\underbrace{0,\ldots,0,1}_{l},0,\ldots,0)$.
There are
$\binom{m+d}{d}-\binom{m+d-l}{d}$ elements of $\Omega(d,m)$ that are lexicographically at least $e_l$.
Among these, exactly $\min(l,m-k)$ belong to $\CC(k)$. Hence
$w'_{F_k(l)}=e_l$. Therefore, if $
F_k(l-1)<s(k)\leq F_k(l)
$, then the first $l-1$ coordinates of $w'_{s(k)}$ are zero and its $l$-th coordinate is non-zero.
Thus $l$ is the smallest non-zero index of $w'_{s(k)}$.
\end{proof}

\begin{lemma}\label{lem: r in terms of al}
If $w'_{s(k)}=(\al_1, \dots, \al_m)$ with $l$ being the smallest index such that $\al_l \neq 0$. Then
\[
r= 1+\sum_{i=1}^m \tbinom{m-i+d-\sum_{j=1}^i \al_j}{m-i+1}+\max(m-k-l,0).
\]
\end{lemma}
\begin{proof}
If $s(k)=0$, then $w'_0=(d, 0, \dots, 0)$, so the result follows immediately.

If $s(k) \geq 1$, we have:
\begin{equation*}
\begin{split}
    s(k)
    &=1+|\{w \in \Omega'(k)
    \mid w>_{\lex}w_{s(k)}'\}| \\
    &=1+|\{w \in \Omega(d,m) 
    \mid w>_{\lex}w_{s(k)}'\}|
    -|\{w \in \CC(k) 
    \mid w>_{\lex}w_{s(k)}'\}|.
\end{split}
\end{equation*}
By \cite[Lemma 2.8]{singhal2025conjecture}, we have 
\[
1+|\{w \in \Omega(d,m) \mid w>_{\lex}w_{s(k)}'\}|
=1+\sum_{i=1}^m \tbinom{m-i+d-\sum_{j=1}^i \al_j}{m-i+1}.
\]
By definition, we have 
\[
|\{w \in \CC(k) 
\mid w>_{\lex}w_{s(k)}'\}|
=\min\{l,m-k\}.
\] 
By plugging in $s(k)=r-(m-k)$, we get the desired expression.
\end{proof}

\begin{proposition}\label{prop: fA increase}
With parameters $d,m,r,q$ fixed, let $\max\{0,m-r\} \leq k\leq m-2$. 
Let $l$ be the first non-zero index of $w'_{s(k)}$.
\begin{enumerate}
    \item If $k \geq m-l$, then $\fa(k)=\fa(k+1)=H_r(d,m;q)$.
    \item If $k \leq m-l-1$, then $\fa(k+1)\geq\fa(k)$.
\end{enumerate}
In particular, $\fa$ is a non-decreasing function with respect to $k$.
\end{proposition}
\begin{proof}
By the definition of $\Omega'(k)$ in Equation \ref{eq: omegap}, we see that:
\[\Omega'(k+1)=\Omega'(k) \cup \{(\underbrace{0,\ldots,0,d,}_{m-k}
0,\ldots,0)\}\]
The position of $w'_{s(k)}$ in $\Omega'(k+1)$ is
\begin{align*}
\#\{w \in \Omega'(k+1) \mid w \geq_{\lex} w'_{s(k)}\}
&=\#\{w \in \Omega'(k) \mid w \geq_{\lex} w'_{s(k)}\}
+\mathbbm{1}_{\{(\underbrace{{\scriptstyle 0,\ldots,0,d}}_{m-k},0,\ldots,0)
\geq_{\lex} w'_{s(k)}\}} \\
&=s(k)
+\mathbbm{1}_{\{(\underbrace{{\scriptstyle 0,\ldots,0,d}}_{m-k},0,\ldots,0)
\geq_{\lex} w'_{s(k)}\}}.
\end{align*}
\begin{enumerate}
\item If $l \geq m-k$: then $(\underbrace{0,\ldots,0,d,}_{m-k} 0,\ldots,0)\geq_{\lex}w'_{s(k)}$. Hence the position of $w'_{s(k)}$ in $\Omega'(k+1)$ is $s(k)+1=s(k+1)$. We have that $w'_{s(k+1)}=w'_{s(k)}$. Hence, $\fa(k)=\fa(k+1)$. Repeating the same argument, we obtain
$\fa(k)=\fa(k+1)=\cdots=\fa(m-1)=H_r(d,m;q)$.
\item If $1 \leq l \leq m-k-1$: then $(\underbrace{0,\ldots,0,d,}_{m-k}
  0,\ldots,0)<_{lex}w'_{s(k)}$. Therefore, the position of $w'_{s(k)}$ in $\Omega'(k+1)$ is $s(k)$. Hence, $w'_{s(k+1)}$ is the next element in lexicographical ordering after $w'_{s(k)}$. Write $w'_{s(k)}=(0, \dots, \al_l, \dots, \al_t \dots, 0)$, so that $l$ is the smallest non-zero index and $t$ is the largest non-zero index. 
  Write $s=\deg(w'_{s(k)})$, so that $s \leq d$. We see that
\[
w'_{s(k+1)}=
\begin{cases}
(0,\dots,\al_l,\dots,\al_t-1,
d-s+1,0,\dots,0),
& t\neq l \text{ or } \al_l\ge 2,\ t\le m-1,\\

(0,\dots,\al_l,\dots,\al_m-1),
& t=m,\\

(\underbrace{0,\ldots,0}_l, d-1,1,0, \ldots,0),
& t=l, \ \al_l=1,\ l\le m-k-2,\\

(\underbrace{0,\ldots,0}_l,d,0,\dots,0),
& t=l,\ \al_l=1,\ l=m-k-1.\\
\end{cases}
\]
  
  Then we have:
  \begin{equation*}
      \begin{split}
          \fa(k) &=q^k\sum_{j=l}^{m-k-1}\al_jd^{m-k-1-j} \cdot d +\sum_{j=m-k}^{t-1}\al_j q^{m-j}+\al_tq^{m-t}, \\
          \fa(k+1)&=q^k\sum_{j=l}^{m-k-1}\al_jd^{m-k-1-j} \cdot q +\sum_{j=m-k}^{t-1}\al_j q^{m-j}+(\al_t-1)q^{m-t}+(d-s+1)q^{m-t-1}.\\
      \end{split}
  \end{equation*}
From which we have:
\begin{enumerate}
\item If $t=l$:
\[
\fa(k+1)-\fa(k)
=
\begin{cases}
d^{m-k-2-l}q^k
\left[\al_l d(q-d)-(\al_l-1)q\right],
& \al_l\ge 2,\ t\le m-k-2,\\

(\al_l-1)(q-d-1)q^k,
&
\al_l \ge 2, t=m-k-1,\\

d^{m-k-3-l}q^k
\left[d^2(q-d)-(d-1)q\right],
& \al_l=1,\ t\le m-k-3,\\

\left[(d-1)q+1-d^2\right]q^k,
&  \al_l=1,\ t=m-k-2,\\

0,
&  \al_l=1,\ t=m-k-1.
\end{cases}
\]
\item If $t \geq l+1$:
\[
\fa(k+1)-\fa(k)
=
\begin{cases}
d^{m-k-2-t}q^k
\left[
(q-d)\sum_{j=l}^{t}\al_jd^{t+1-j}
-q\left(s-1\right)
\right],
& t\le m-k-2,\\

q^k
\left[
(q-d)\sum_{j=l}^{t}\al_jd^{t-j}
+d-s+1-q
\right],
& t=m-k-1,\\

q^{m-t-1}
\left[
(q-d)q^{k-m+t+1}
\sum_{j=l}^{m-k-1}\al_jd^{m-k-1-j}
+d-s+1-q
\right],
& m-k\le t\le m-1,\\

(q-d)q^k
\sum_{j=l}^{m-k-1}\al_jd^{m-k-1-j}
-1,
& t=m.\\
\end{cases}
\]
\end{enumerate}
  \end{enumerate}
From the above, we see that in the following cases, we have $\fa(k+1)-\fa(k) \geq q^k$:
\begin{enumerate}
    \item If $t=l$:
    \begin{itemize}
        \item $ l \leq m-k-3$.
        \item $l=m-k-2$ and $\al_l \geq 2$ or $\al_l=1$ and $q \geq d+2$. Otherwise, the difference is $0$.
        \item $l=m-k-1$ and $\al_l \geq 2$ and $q \geq d+2$. Otherwise, the difference is $0$.
        \end{itemize}
        
        \item If $t \geq l+1$:
        \begin{itemize}
            \item $ t \leq m-k-1$.
            \item $m-k \leq t \leq m-1$ except for $l=m-k-1,\al_l=1,q=d+1$,  \\
            when the difference is $q^k-sq^{m-t-1}\geq q^{m-t-1}>0$ since $s\leq d$.
            \item $t=m$ except for $l=m-k-1,\al_l=1,q=d+1$, when the difference is $q^k-1$.
        \end{itemize}
\end{enumerate}
\end{proof}

\section{The affine algebraic conjecture and EGH conjecture}\label{sec: EGH and affine}

In this section we will show that $\EGH(d^{m-k}, q^k)_m$ implies Conjecture \ref{conj: affine conjecture 2}. We focus on standard graded Artinian algebras over $\kk$, where $\kk$ is an algebraically closed field. For a standard graded Artinian $\kk$ algebra $A$,
\begin{itemize}
\item  We use $A_i$ to denote the $i$-th graded piece of $A$.
\item We use $h_{A}$ for the Hilbert function of $A$, so $h_{A}(i)=\dim_{\kk}(A_i)$
\end{itemize}
\begin{proposition}\label{prop: EGH implies conjecture 2}
Let $A=S/I$ be a standard graded Artinian algebra over $\kk$, such that $h_A(1)=m$ and $I$ contains a regular
sequence of degree $\dqmk$. If $\EGH(d^{m-k},q^k)_m$ holds for $A$, then
the conclusion in Conjecture \ref{conj: affine conjecture 2} is true. That is,
\[
\dim_{\kk}(A_{\leq d}) \leq \tbinom{m+d}{d}-r
\implies
\dim_{\kk}(A) \leq \fA.
\]
\end{proposition}
Assuming $\EGH(d^{m-k},q^k)_m$, let $\Lambda:=\Lambda(A)$, $\Omega':=\Omega'(d,m,k;q)$, and $\B:=\B\dmkq$. We have $\dim_{\kk}(A_{\leq d})=|\Lambda\cap\Omega'|$ and 
$\dim_{\kk}(A)=|\Lambda|$. So, we just need to show that for all the subsets $\Lambda \subseteq \B$ that are closed and compressed, 
\begin{equation*}
|\Lambda \cap \Omega'| \leq \tbinom{m+d}{d}-r \implies |\Lambda| \leq \fA.
\end{equation*}

To prove this proposition, we consider a specific type of subset $\Lambda_s$.
\begin{definition}
For $1 \leq s \leq |\Omega'|=\binom{m+d}{d}-(m-k)$, let $w'_s$ denote the $s$-th lexicographically largest element in $\Omega'$. We define
\[
\Lambda_s=\{w \in \B \mid w <_{\lex}w'_s\}.
\] 
We use the convention that $\Lambda_0=\B$.
\end{definition}
Notice that $\Lambda_s$ are closed and compressed. Moreover, if we take $s=r-(m-k)$, we get that $|\Lambda_{r-(m-k)} \cap \Omega'|=\binom{m+d}{d}-r$ and $|\Lambda_{r-(m-k)}|=\fA$, so it achieves equality in Conjecture~\ref{conj: affine conjecture 2}. 
\begin{definition}
    Let $1 \leq N \leq |\B|=d^{m-k}q^k$. We define:
    \[\Sigma_N:=\{w \in \B \mid w\text{ is among the lex-smallest $N$ elements in $\B$}\}.\]
    Notice that if we let $N=|\Lambda_s|$, then $\Lambda_s=\Sigma_N$.
\end{definition}
So, Proposition \ref{prop: EGH implies conjecture 2} is equivalent to the following.
\begin{proposition}\label{prop: Ls is optimal}
Let \(\Lambda\subseteq \B\) be closed and compressed.
If $|\Lambda \cap \Omega'|\leq \binom{m+d}{d}-r$, then $|\Lambda| \leq |\Lambda_{r-(m-k)}|$.
\end{proposition}
To prove this proposition, we first need a lemma, which is the consequence of \cite[Corollary 3]{clements1969generalization}.

\begin{proposition}\label{prop: compression lemma}
Let $\Lambda \subseteq \B$ be closed and compressed. If $|\Lambda|=N$, then $|\Lambda \cap \Omega'| \geq |\Sigma_N \cap \Omega'|$. In other words, $\Sigma_N$ are the subsets of $\B$ of size $N$ that minimizes the degree $\leq d$ part. 
\end{proposition}

\begin{proof}
Given $\Lambda$, if $\Lambda \neq \Sigma_N$, we show that we can construct $\Lambda'$ such that:
\begin{enumerate}
\item $|\Lambda'\cap \Sigma_N|=|\Lambda\cap \Sigma_N|+1 $,
\item $|\Lambda' \cap \Omega'|\leq |\Lambda \cap \Omega'|$,
\item $\Lambda'$ is closed and compressed with $|\Lambda'|=|\Lambda|$.
\end{enumerate}
Suppose $\Lambda \neq \Sigma_N$. Then, let $x=\min_{\lex}(\B \backslash \Lambda)$ and $y=\max_{\lex}(\Lambda)$. 
We will take $\Lambda'=\Lambda \backslash \{y\} \cup \{x\}$. 
\begin{enumerate}
\item We claim $x <_{\lex}y$. Since $\Lambda \neq \Sigma_N$, we know $\B \backslash \Lambda \neq \B \backslash \Sigma_N$. Hence, 
\[
x
=\min_{\lex}(\B \backslash \Lambda)
<_{\lex} \min_{\lex}(\B \backslash \Sigma_N).
\] 
So $x$ is among the $N$ lexicographically smallest elements of $\B$, that is $x \in \Sigma_N$. On the other hand, 
\[
y
=\max_{\lex}(\Lambda)
>_{\lex}\max_{\lex}(\Sigma_N).
\] 
So $y \not \in \Sigma_N$ and hence $x<_{\lex}y$. Consequently, $|\Lambda' \cap \Sigma_N|=|\Lambda \cap \Sigma_N|+1$.

\item We claim that $\deg(x)>\deg(y)$. Since $\Lambda$ is compressed, if $\deg(x)=\deg(y)$ and $x<_{\lex}y$, one must have $x \in \Lambda$, contradiction. 

Now assume for the sake of contradiction that $\deg(x)<\deg(y)$. We will construct $c \in \B$ such that:
\begin{itemize}
    \item $c \preccurlyeq y$,
    \item $x<_{\lex}c$,
    \item $\deg(c)=\deg(x)$.
\end{itemize}
If such $c$ can be constructed, then we will get a contradiction. Indeed, since $\Lambda$ is closed, $c \preccurlyeq y \implies c \in \Lambda$. Since $\Lambda$ is compressed, $\deg(c)=\deg(x)$ and $x<_{\lex}c$ implies that $x \in \Lambda$.

Now we construct the element $c$. Let $D=\deg(y)-\deg(x)$. Let $l$ be the first index that $x_l \neq y_l$, so $x_l<y_l$. We let $c_i=x_i=y_i$ for $1 \leq i \leq l-1$. For the index from $l$, we break into two cases.
\begin{enumerate}
    \item If $D<y_l-x_l$. Then, let $c_l=y_l-D$ and $c_j=y_j$ for $j \geq l+1$. Then $c_l>_{lex}x_l$, $\deg(c)=\deg(x)$ and $c\preccurlyeq y$.
    \item If $D \geq y_l-x_l$. Then, $\sum_{j>l}y_j \geq \sum_{j>l}x_j$. Denote by $X_l$ the sum $\sum_{j>l}x_j$. Notice that $X_l\geq 1$. Indeed, if $X_l=0$, then $x_j=0$ for all
$j>l$. Since $x_i=y_i$ for $i<l$ and $x_l<y_l$, we would have
$x<y$ coordinatewise. Since $y\in\Lambda$ and $\Lambda$ is closed,
this would imply $x\in\Lambda$, a contradiction.

We take $c_l=x_l+1$. Then, we need to find $\{c_j\}_{j>l}$ such that $c_j \leq y_j$ and $\sum_{j>l}c_j=X_l-1$. We can first take $c_j=y_j$ and then subtract from each coordinate until the sum $\sum_{j>l}c_j$ is equal to $X_l-1$.
    \end{enumerate}
    Therefore, $\deg(x)>\deg(y)$. Since $\Lambda'=\Lambda \backslash\{y\} \cup \{x\}$, the degree $\leq d$ elements can not increase. Hence, $
    |\Lambda' \cap \Omega'| \leq |\Lambda \cap \Omega'|$.
    \item We will show that $\Lambda'=\Lambda \backslash \{y\} \cup \{x\}$ is closed and compressed.
    \begin{itemize}
        \item First we show that $\Lambda'$ is closed. Consider $w \in \Lambda'$, and $w'\prec w$. 
        If $w \in \Lambda \backslash\{y\}$, then $w' \in \Lambda$ since $\Lambda$ is closed. 
        Moreover, since $w'<_{\lex} w <_{\lex}y$, we have $w' \in \Lambda \backslash \{y\} \subseteq \Lambda'$. 
        On the other hand, if $w=x$, then in particular $w'<_{\lex}x$. Since $x=\min_{\lex}(\B \backslash \Lambda)$, this implies $w' \in \Lambda$. This shows that $\Lambda'$ is closed. 
        
        \item Next we show that $\Lambda'$ is compressed. We just need to consider the layer of $\Lambda'$ of $\deg(y)$ or $\deg(x)$. 
        Since $y$ is the lex-max element in $\Lambda$, in particular it is the lex-max in $\Lambda(\deg(y))$, so removing $y$ does not impact the property of being compressed. 
        Since $x$ is the lex-min element in $\B\setminus\Lambda$, in particular it is the lex-min in $(\B\setminus \Lambda)(\deg(x))$, so adding $x$ does not impact the property of being compressed either.
        Thus $\Lambda'$ is compressed.  
    \end{itemize}
\end{enumerate}
Starting from $\Lambda \neq \Sigma_N$ and 
repeatedly applying this operation finitely many times gives \(\Sigma_N\). At every step $|\Lambda' \cap \Omega'|$ does not increase. Hence
$|\Sigma_N\cap \Omega'| \leq |\Lambda\cap \Omega'|$.
This proves the proposition.
\end{proof}

\begin{proof}[Proof of Proposition \ref{prop: Ls is optimal}]
If $r=m-k$, then $\Lambda_{r-(m-k)}=\B$ so $|\Lambda| \leq |\Lambda_{r-(m-k)}|$. 

Assume $r>m-k$. Given $\Lambda$, let $N=|\Lambda|$ and suppose $N>|\Lambda_{r-(m-k)}|$. Then we have $w'_{r-(m-k)}\in \Sigma_N$. Since $w'_{r-(m-k)} \in \Omega'$, we have
$|\Sigma_N\cap \Omega'|
\geq |\Lambda_{r-(m-k)}\cap \Omega'|+1$. Hence, by Proposition \ref{prop: compression lemma}, we have 
\[
|\Lambda \cap \Omega'| 
\geq |\Lambda_{r-(m-k)}\cap \Omega'|+1
=\tbinom{m+d}{d}+1-r
\] 
which contradicts the assumption that $|\Lambda \cap \Omega'| \leq \binom{m+d}{d}-r$.
This proves the proposition.
\end{proof}

\section{Equality construction of the affine geometric conjecture}\label{sec: affine conjecture equality}
In this section, we want to show that the conjectured formula $\fA$ is at least a lower bound for the quantity $\eA$.

\begin{proposition}\label{prop :affine-geometric-equality}
Given $d,m\geq 1$, $0\leq k\leq m-1$, $m-k\leq r \leq \binom{m+d}{d}$ and $q\geq d+1$,
we have
\[
\fA \leq \eA.
\]
\end{proposition}
To show this proposition, we fix the parameters $d,m,k,q$ and $r$. Suppose $w'_{r-(m-k)}=(\al_1,\dots, \al_m)$.
We want to construct $W \subseteq A_{\leq d}(m,\Fq)$ such that:
\begin{enumerate}
\item $|Z(W)(\Fq)|=\fA$,
\item $\dim(Z(W)) \leq k$,
\item $\dim_{\Fq}(W)=r$.
\end{enumerate}
This will imply that $\fA$ is at least a lower bound for the quantity $\eA$. Our construction closely follows that of \cite[Section 3]{lin2026largest}. 

We enumerate elements in $\Fq$ by $\{a_1,a_2, \dots, a_q\}$. Reorder them so that $a_d=0$.

For $1 \leq i \leq m-k$ define
\[f_i=\prod_{k=1}^{d}(x_i-a_k).\]
In particular, notice that $f_i$ is a polynomial in $x_i$, and that $x_i \mid f_i$.

For $1 \leq j \leq m$ define
\[g_j=\prod_{k=1}^{\alpha_j}(x_j-a_k).\]
In particular, if $\al_j=0$ then $g_j=1$.

For $1 \leq i \leq m-1$, let 
\[
W_i
=g_1 \dots g_i x_i
\Fq[x_i \dots, x_m]_{\leq d-1-\sum_{k=1}^{i}\al_k}.
\] 
Notice that if $i<l$, then $W_i=x_i\Fq[x_i \dots, x_m]_{\leq d-1}$. 
For $i=m$, let
\[
W_m
=g_1 \dots g_m
\Fq[x_m]_{\leq d-\sum_{k=1}^m \al_k}.
\]
We let $W=W_1+W_2+ \dots +W_m + \langle f_{l+1}, \dots, f_{m-k}\rangle$. By convention, if $l>m-k-1$, then the vector space $\langle f_{l+1}, \dots, f_{m-k}\rangle=\{0\}$. We want to show the three properties mentioned at the beginning of this section.

\begin{lemma}\label{lem: V(W) small dim}
\[\dim(Z(W)) \leq k.\]
\end{lemma}
\begin{proof}
We have $Z(W)\subseteq Z(x_1,\dots,x_{l-1}, f_l, \dots,f_{m-k})$, which is $k$ dimensional.
\end{proof}

\begin{lemma}\label{lem: V(W) has the right size}
For $l \leq i \leq m$, let $Y_i=\prod_{c=1}^m L_{i,c}$, where \begin{equation}\label{eq: lc}
   L_{i,c}=\begin{cases}
   \{a_d\} & \text{ if $c \leq i-1$}\\
   \{a_1, \dots, a_{\al_i}\} &\text{ if $c=i$}\\
   \{a_1, \dots, a_d\} & \text{ if $i+1 \leq c \leq m-k$}\\
   \{a_1, \dots, a_q\}  & \text{ if $c \geq i+1, c \geq m-k+1$}\\
    \end{cases}
\end{equation}
And let $Y=\bigcup_{i=l}^m Y_i$. Then, $Y=Z(W)(\Fq)$.
\end{lemma}

    Notice that the $Y_i$ are pairwise disjoint, and $|Y_i|=\begin{cases}
    \al_id^{m-k-i}q^k & \text{if $i \leq m-k$}\\
\al_iq^{m-i}  & \text {if $i \geq m-k+1$}\end{cases}$. \\ 
Hence, Lemma \ref{lem: V(W) has the right size} implies that $|Z(W)(\Fq)|=\sum_i |Y_i|=\fA$.

\begin{proof}
    Notice that $$Y_i=\begin{cases}
        Z(x_1, \dots, x_{i-1}, g_i, f_{i+1}, \dots, f_{m-k})(\Fq)&\text{ if $i \leq m-k-1$}\\
        Z(x_1, \dots, x_{i-1},g_i)(\Fq) &\text{ if $i \geq m-k$}
    \end{cases}$$
\begin{enumerate}
    \item Show that $Y \subseteq Z(W)(\Fq)$: Since each point in $Y_i$ has $\Fq$-rational coordinates, it suffices to show that $Y_i \subseteq Z(W_j)$ for $1 \leq j \leq m$ and $Y_i \subseteq Z(f_j)$ for $l+1 \leq j \leq m-k$.
    \begin{itemize}
        \item Show that $Y_i \subseteq Z(f_j)$: Given $l+1 \leq j \leq m-k$, it suffices to show that $Y_{i,j} \subseteq \{a_1, \dots, a_d\}$. From Equation \ref{eq: lc}, this is true since $j \leq m-k$.
        

        \item Show that $Y_i \subseteq Z(W_j)$:
        \begin{enumerate}
            \item If $j \leq i-1$: then in particular $j \leq m-1$. Hence we have $L_{i,j}=\{a_d\}=\{0\}$. Since everything in $W_j$ is a multiple of $x_j$, every polynomial in $W_j$ vanishes on $Y_i$.
            \item If $j \geq i$: Then, everything in $W_j$ is a multiple of $g_i$. By Equation \ref{eq: lc}, $L_{i,i}=\{a_1, \dots, a_{\al_i}\}$ and $g_i=\prod_{s=1}^{\al_i}(x_i-a_s)$. Hence, $g_i$ vanishes on $Y_i$, and so $W_j$ vanishes on $Y_i$.
        \end{enumerate}
        This shows that $Y \subseteq Z(W)(\Fq)$.
    \end{itemize}
    \item Show that $Z(W)(\Fq) \subseteq Y$: Let $P=(y_1, \dots, y_m) \in \A^m(\Fq)$. Assume that $P \in Z(W)$. We want to show that $P \in Y$.

    Since $g_lg_{l+1}\cdots g_m\in W_m$, there is some $i\geq l$ such that $g_i$ vanishes on $P$. Let $i_0$ be the smallest such index. Since $g_{i_0}(y_{i_0})=0$, we see that $y_{i_0}\in\{a_1,\dots,a_{\al_{i_0}}\}$.

We claim that $y_j=a_d=0$ for every $j<i_0$. If $j<l$, then $x_j\in W_j$, so $x_j$ vanishes on $P$, and hence $y_j=0$. If $l\leq j<i_0$, then $g_lg_{l+1}\cdots g_jx_j\in W_j$. By the minimality of $i_0$, none of $g_l,\dots,g_j$ vanishes on $P$. Hence, $x_j$ must vanish on $P$, so $y_j=0=a_d$.

Finally, for $i_0+1\leq i\leq m-k$, since $f_i\in W$, we have $f_i(P)=0$, and therefore $y_i\in\{a_1,\dots,a_d\}$. For the remaining coordinates, we know $y_i\in\Fq$. Hence, $P\in Y_{i_0}\subseteq Y$.  
\end{enumerate}
This finishes the proof that $Y=Z(W)(\Fq)$, and gives us that $|Z(W)(\Fq)|=\fA$.
\end{proof}
\begin{lemma}\label{lem: dim eq r}
    We have that $\dim_{\Fq}(W)=r$.
\end{lemma}
\begin{proof}
We want to show that $W=W_1+\dots+W_m+\langle f_{l+1},\dots,f_{m-k}\rangle$ is a direct sum. 

Equip the monomials with the lexicographic order $x_1>\dots>x_m$, and for a nonzero polynomial $F$, let $\operatorname{LM}(F)$ denote its leading monomial with respect to this order.

Suppose that $F_i\in W_i$ for $1\leq i\leq m$ and $u_{l+1},\dots,u_{m-k}\in\Fq$ satisfy
\[
F_1+\dots+F_m+\sum_{j=l+1}^{m-k}u_jf_j=0.
\]
We will show that all $F_i$ and $u_j$ are zero. For $1\leq i\leq m-1$, if $F_i\neq 0$, write $F_i=(x_i\prod_{j=1}^i g_j)h_i$ for some nonzero $h_i\in\Fq[x_i,\dots,x_m]_{\leq d-1-\sum_{j=1}^i\al_j}$. Since $\operatorname{LM}(g_j)=x_j^{\al_j}$, we have
\[
\operatorname{LM}(F_i)=x_i\Big(\prod_{j=1}^i x_j^{\al_j}\Big)\operatorname{LM}(h_i).
\]
Similarly, if $F_m\neq 0$, then $\operatorname{LM}(F_m)=(\prod_{j=1}^m x_j^{\al_j})\operatorname{LM}(h_m)$ for some nonzero $h_m\in\Fq[x_m]_{\leq d-\sum_{j=1}^m\al_j}$.

Suppose $i<j$ and both $F_i,F_j$ are nonzero. The exponent of $x_i$ in $\operatorname{LM}(F_i)$ is at least $\al_i+1$, whereas the exponent of $x_i$ in $\operatorname{LM}(F_j)$ is exactly $\al_i$, since the remaining factor in $F_j$ only involves $x_j,\dots,x_m$. Hence, $\operatorname{LM}(F_i)\neq\operatorname{LM}(F_j)$.

Next, for $l+1\leq j\leq m-k$, we have $\operatorname{LM}(f_j)=x_j^d$. We claim that this cannot equal $\operatorname{LM}(F_i)$ for any nonzero $F_i\in W_i$. 
\begin{itemize}
\item If $i<j$, then $\operatorname{LM}(F_i)$ has a positive exponent in $x_i$, whereas $x_j^d$ does not. 
\item If $i\geq j$, then $\operatorname{LM}(F_i)$ has a positive exponent in $x_l$, since $\al_l>0$, whereas $x_j^d$ has exponent zero in $x_l$ since $j>l$. 
\end{itemize}
Hence, $\operatorname{LM}(F_i)\neq\operatorname{LM}(f_j)$ for every $i$ and $j$. Moreover, the leading monomials $x_j^d$ of the different $f_j$ are distinct.

Therefore, if some $F_i$ or $u_j$ is nonzero, the largest leading monomial among the nonzero terms in the above relation occurs exactly once and cannot cancel. This is a contradiction. Hence, all $F_i$ and $u_j$ are zero, and therefore
\[
W=W_1\oplus\dots\oplus W_m\oplus\langle f_{l+1},\dots,f_{m-k}\rangle.
\]
Hence, $\dim(W)=\sum_{i=1}^m\dim(W_i)+\max(m-k-l,0)$. Notice that
\[
\dim(W_i)=\tbinom{m-i+d-\sum_{j=1}^i\al_j}{m-i+1}
\]
for $1\leq i\leq m-1$, and $\dim(W_m)=1+d-\sum_{j=1}^m\al_j$. Hence,
\[
\dim(W)=1+\sum_{i=1}^m\tbinom{m-i+d-\sum_{j=1}^i\al_j}{m-i+1}+\max(m-k-l,0).
\]
By Lemma \ref{lem: r in terms of al}, this is equal to $r$. Hence, $\dim(W)=r$.
\end{proof}
Since we are able to construct $W$ with dimension $r$, $\dim(Z(W))\leq k$ and $|Z(W)(\Fq)|=\fA$, this shows that the conjectured bound is achievable, and hence proves that $\eA \geq \fA$.

\section{Projective conjecture}\label{sec: projective conjecture}

Throughout this section, fix $d,m\geq 1$, $q\geq d+1$, $0\leq k\leq m-1$, and $m-k\leq r\leq \binom{m+d}{d}$.
Write $w'_{r-(m-k)}(d,m,k;q)=(\al_1,\ldots,\al_m)$,
let $l$ be the smallest index such that $\al_l\neq 0$, and set
\[
s=\min(m-l-1,k-1).
\]

Recall that for $W\subseteq S_d(m,\Fq)$, we denote its vanishing locus in $\PP^m$ by $V(W)$, and
\[
\eP
=
\max\Bigl\{
|V(W)(\Fq)|:
W\subseteq S_d(m,\Fq),\
\dim_{\Fq}W=r,\
\dim V(W)\leq k
\Bigr\}.
\]
Also recall that $\pi_j(q)=|\PP^j(\Fq)|=\frac{q^{j+1}-1}{q-1}$ for $j\geq 0$, while $\pi_j(q)=0$ for $j<0$. In Conjecture \ref{conj:projective-geometric}, we give the conjectured value of $\eP$ as
\[
\fP=\fA+\pi_s(q).
\]
In this section we show that this conjectured formula is a lower bound for $\eP$.

\begin{proposition}\label{prop: equal proj}
\[
\fP\leq \eP.
\]
\end{proposition}

To prove this proposition, we construct $\Wb\subseteq S_d(m,\Fq)$
such that:
\begin{itemize}
    \item $|V(\Wb)(\Fq)|=\fP$;
    \item $\dim V(\Wb)\leq k$;
    \item $\dim_{\Fq}(\Wb)=r$.
\end{itemize}

For a polynomial $f \in A_{\leq d}(m,\Fq)$, recall that the degree-$d$ homogenization of $f$ with respect to $X_0$ is \[f^h:=X_0^{d}f(\tfrac{X_1}{X_0}, \dots, \tfrac{X_m}{X_0})\]
so $f^h \in S_d(m,\Fq)$.

We take $\Wb$ to be the degree-$d$ homogenization of $W$ from Section \ref{sec: affine conjecture equality} with respect to $X_0$. That is, \\$\Wb=\langle f^h \mid f \in W \rangle$. More precisely, we take
\[F_i=\prod_{k=1}^{d}(X_i-a_kX_0), \qquad G_j=\prod_{k=1}^{\al_j}(X_j-a_kX_0).\]
For $1 \leq i \leq m-1$, let \[\Wb_i=G_1\dots G_iX_i\Fq[X_0,X_i \dots, X_m]_{d-1-\sum_{k=1}^{i}\al_k}.\] 
If $i<l$, then $\Wb_i=X_i\Fq[X_0,X_i \dots, X_m]_{d-1}$. For $i=m$, let
\[
\Wb_m=G_1 \dots G_m\Fq[X_0,X_m]_{d-\sum_{k=1}^m \al_k}.
\]
We let $\Wb=\Wb_1+\Wb_2+ \dots +\Wb_m + \langle F_{l+1}, \dots, F_{m-k}\rangle$. Notice that for $F \in \Wb$, we have $F(1,x_1, \dots, x_m)\in W$.

Further notice that $\dim_{\Fq}(\Wb)=\dim_{\Fq}(W)=r$, since degree-$d$ homogenization is a vector space isomorphism from $A_{\leq d}(m,\Fq)$ to $S_d(m,\Fq)$.

Moreover, $V(\Wb)\cap U(X_0)\cong Z(W)$. Hence, by Section~\ref{sec: affine conjecture equality}, $\dim(V(\Wb)\cap U(X_0))\leq k$
and \\ $|(V(\Wb)\cap U(X_0))(\Fq)|=\fA$. It remains to determine $V(\Wb)$ on the hyperplane $V(X_0)$.

\begin{lemma}\label{lem: point at infinity}
Let $t=\max(l,m-k)$. Then,
    \[V(\Wb) \cap V(X_0)= \bigcap_{j=0}^t V(X_j).\]
\end{lemma}
\begin{proof}
Setting $X_0=0$, we see that
\[F_i(X_0=0)=X_i^d, \qquad G_j(X_0=0)=X_j^{\al_j}.\]
For any $P \in V(\Wb) \cap V(X_0)$, say $P=[0:Y_1: \dots:Y_m]$. We consider the following cases:
\begin{itemize}
    \item If $1\leq i<l$, then $\Wb_i|_{X_0=0}$ contains $X_i^d$. Since $\Wb_i$ vanishes on $P$, we must have $Y_i=0$.
    \item If $i=l$, then if $\al_l<d$, $\Wb_l|_{X_0=0}$ contains $X_l^d$, so $Y_l=0$. If $\al_l=d$, then $\al_j=0$ for every $j\neq l$, and $\Wb_m|_{X_0=0}$ contains $X_l^d$, so again $Y_l=0$.
    \item If $l<m-k$, then for $l<j\leq m-k$, since $F_j$ vanishes on $P$, we must have $Y_j=0$.
\end{itemize}
Hence, $Y_i=0$ for $1\leq i\leq t=\max(l,m-k)$, and therefore $P \in \bigcap_{j=0}^t V(X_j)$.


Conversely, if $P \in V(X_0)$ has $Y_i=0$ for $1 \leq i \leq t$, then $\Wb_i$ vanishes on $P$ for $1 \leq i \leq l$ and $F_j$ vanishes on $P$ for $l+1\leq j\leq m-k$. 
For $l+1 \leq i \leq m$, notice that $G_l$ divides all polynomials in $\Wb_i$, so $\Wb_i$ vanishes on $P$. 
This shows that $\bigcap_{j=0}^t V(X_j) \subseteq V(\Wb) \cap V(X_0)$, which proves the desired equality.
\end{proof}

From Lemma \ref{lem: point at infinity}, we see that $\dim(V(\Wb) \cap V(X_0)) \leq k-1$ and $|(V(\Wb) \cap V(X_0))(\Fq)|=\pi_s(q)$. Hence, the $\Wb$ we constructed satisfies $\dim_{\Fq}(\Wb)=r, \dim(V(\Wb)) \leq k$, and 
\[
|V(\Wb)(\Fq)|
=|Z(W)(\Fq)| + |\PP^s(\Fq)|
=\fA+\pi_s(q).
\] 
Hence, we show that the equality can be achieved, which shows that $\fP$ is a lower bound for $\eP$.

\section{$m-2$ dimensional projective conjecture implies B-D-G}\label{sec: m-2 proj implies B-D-G}
In this section, we show that Conjecture \ref{conj:projective-geometric} in $k=m-2$ implies Conjecture \ref{Conj: complete GDC}. That is, we want to show the following statement.
\begin{customprop}{\ref{prop: projective m-2 implies BDG}}
Suppose $d\geq 1, m \geq 2,2\leq r\leq \binom{m+d}{d}$ and $q\geq d+1$. 
If we know that for each $1\leq d'\leq d$ satisfying $r\leq \binom{m+d'}{d'}$,
\[
e^{\PP}_r(d',m,m-2;q) 
=f^{\PP}_r(d',m,m-2;q),
\]
then we have that 
\[
e^{\PP}_r(d,m,m-1;q) 
=f^{\PP}_r(d,m,m-1;q).
\]
\end{customprop}
We follow the strategy of \cite[Theorem 1.5]{lin2026largest}.
We first recall several results that will be used in the proof.

\begin{lemma}\cite[Lemma 38]{singhal2025conjecture}\label{lemma: d-c}
Given $m\geq 1$, $1\leq c\leq d-1$ and $1\leq r\leq \binom{m+d-c}{d-c}$, we have
$$H_r(d-c,m;q)+cq^{m-1}\leq H_r(d,m;q).$$
\end{lemma}
\begin{lemma}
\cite[Lemma 2.9]{beelen2018maximum}\label{lem: linear factor}
Suppose $q \geq d$ and $\dim_{\Fq}(W)=r$. If there exists a linear form $L$ defined over $\Fq$ such that $L \mid F$ for any $F \in W$, then we have
$$|V(W)(\Fq)| \leq H_r(d-1,m;q)+\pi_{m-1}(q).$$
\end{lemma}
\begin{lemma}
\cite[Theorem 1.2]{homma2013elementary}\label{lem: no linear divisor hypersurface}
For $d\geq 2$, consider $G\in S_d(m,\Fq)$, $G\neq 0$. Assume that $G$ has no linear factors over $\Fq$. Then, we have
$$|V(G)(\Fq)| \leq (d-1)q^{m-1}+dq^{m-2}+\pi_{m-3}(q).$$
\end{lemma}
\begin{lemma}\cite[Lemma 5.1]{lin2026largest}\label{Lem: 2 coprime poly}
Let $W \subseteq S_d(m, \Fq)$. If $q \geq d+1$ and $\gcd(W)=1$, then there is a pair of coprime polynomials in $W$.
\end{lemma}

\begin{proof}[Proof of Proposition \ref{prop: projective m-2 implies BDG}]
Let $W \subseteq S_d(m,\Fq)$ such that $\dim_{\Fq}(W)=r$. Let $g=\gcd(W)$, and let $e=\deg(g)$. Write $w_r(d,m)=(\beta_1,\dots,\beta_{m+1})$, and let $l$ be the smallest index such that $\beta_l\neq 0$.
\begin{enumerate}
\item If $e=0$, then $\gcd(W)=1$. By Lemma \ref{Lem: 2 coprime poly}, we can find a regular sequence of length $2$ in $W$. Hence, $\dim(V(W)) \leq m-2$. By  Proposition \ref{prop: fA increase} and the assumption that $e^{\PP}_r(d,m,m-2;q) 
=f^{\PP}_r(d,m,m-2;q)$, we see that 
\begin{equation*}
    \begin{split}
    |V(W)(\Fq)| &\leq f^{\PP}_r(d,m,m-2;q)\\ 
    &= \fa_r(d,m,m-2;q)+\pi_{\min(m-l-1,m-3)}(q) \\
    & \leq   \fa_r(d,m,m-1;q)+\pi_{m-l-1}(q) \\
      &=f^{\PP}_r(d,m,m-1;q).
    \end{split}
\end{equation*}
   
\item If $e \geq 1$, then write $W=gW_1$, where $\gcd(W_1)=1$. Since $\dim_{\Fq}(W_1)=r$ and $W_1 \subseteq S_{d-e}(m,\Fq)$, we have $r \leq \binom{m+d-e}{m} \leq \binom{m+d-1}{m}$, so $l=1$.

We further break into two cases based on whether $g$ has a $\Fq$-linear factor or not.
\begin{itemize}
\item If $g$ has a $\Fq$-linear factor, then, by Lemma \ref{lem: linear factor} and Lemma \ref{lemma: d-c}, we have that
\begin{align*}
\begin{split}
    |V(W)(\Fq)| &\leq H_r(d-1,m;q)+\pi_{m-1}(q) \\
    &\leq H_r(d,m;q)-q^{m-1}+\pi_{m-1}(q)\\
    &=H_r(d,m;q)+\pi_{m-2}(q)=f^{\PP}_r(d,m,m-1;q).
\end{split}
\end{align*}
    
\item If $g$ does not have a $\Fq$-linear factor, then we have $V(W) \subseteq V(W_1) \cup V(g)$. By the assumption that $e^{\PP}_r(d-e,m,m-2;q) 
=f^{\PP}_r(d-e,m,m-2;q)$, we have
\begin{align*}
\begin{split}
|V(W_1)(\Fq)| &\leq f^{\PP}_r(d-e,m,m-2;q)\\
&\leq \fa_r(d-e,m,m-2;q)+\pi_{m-3}(q)\\
&\leq \fa_r(d-e,m,m-1;q)+\pi_{m-3}(q)\\
\end{split}
\end{align*}
Since $g$ has no $\Fq$-linear factor, we have $e\geq 2$. By Lemma \ref{lem: no linear divisor hypersurface} we have that 
\[
|V(g)(\Fq)|
\leq (e-1)q^{m-1}+eq^{m-2}+\pi_{m-3}(q).
\]
Hence we have
\begin{align*}
\begin{split}
    &|V(W)(\Fq)| \leq |V(W_1)(\Fq)|+|V(g)(\Fq)| \\
    & \leq \fa_r(d,m,m-1;q)-eq^{m-1}+\pi_{m-3}(q)+(e-1)q^{m-1}+eq^{m-2}+\pi_{m-3}(q) \\
    &\leq \fa_r(d,m,m-1;q)+(e-1)q^{m-2}+\pi_{m-2}(q)+\pi_{m-3}(q)-q^{m-1}.
\end{split}
\end{align*}
We have that 
\[
(e-1)q^{m-2}+\pi_{m-3}(q) 
\leq q^{m-2}(e+\frac{1}{q-1}) 
\leq  q^{m-2}\cdot q 
=q^{m-1}.
\]
Hence we have $|V(W)(\Fq)|\leq \fa_r(d,m,m-1;q)+\pi_{m-2}(q) $, as desired. \qedhere
\end{itemize}
\end{enumerate}
\end{proof}

\section{A special case of the projective conjecture}\label{sec: codim 2}

In this section we study Conjecture \ref{conj:projective-geometric} where $k=m-2$ and $r=m-k=2$. 

Suppose $W \subseteq S_d(m,\Fq)$ such that $\dim_{\Fq}(W)=2$ and $\dim(V(W))=m-2$. Let $F_1, F_2$ be a basis of $W$ over $\Fq$, so $W=\langle F_1,F_2 \rangle$.
Then, $( F_1,F_2 )$ is a regular sequence in $S(m,\Fq)$, and $V(W)$ is pure of codimension $2$.

Given  $W\subseteq S_d(m,\Fq)$ and a linear form $L\in S_1(m,\Fq)$, Beelen, Datta and Ghorpade in \cite{beelen2018maximum} define 
\begin{align*}
&t_W(L)=\dim(W\cap LS_{d-1}(m,\Fq)),
&
&t_W=\max\{t_W(L)\mid L\in S_1(m,\Fq)\}.
\end{align*}
Note that in our case, with $W=\langle F_1,F_2 \rangle$, we have $0\leq t_W\leq 1$. Indeed, if $t_W=2$, then $F_1$ and $F_2$ have a common linear factor, so $(F_1,F_2)$ is not a regular sequence.

The main result of this section is the following theorem:
\begin{customthm}{\ref{prop: r=m-k,k=m-2}}
Consider $m\geq 2$, $d\geq 1$ and $q\geq d+1$.
For every $W=\langle F_1,F_2 \rangle \subseteq S_d(m,\Fq)$ with $\dim(W)=2$ and $\dim(V(W))=m-2$, we have
\[
|V(W)(\Fq)| 
\leq d^2q^{m-2}+\pi_{m-3}(q).
\] 
Thus, \[e^{\mathbb{P}}_2(d,m,m-2;q) = f^{\mathbb{P}}_2(d,m,m-2;q).\]

Moreover, if $m \geq 3$, then the equality holds only when $t_W=1$ and $V(W)$ does not contain any $\Fq$-linear component with multiplicity $\geq 2$.
\end{customthm}

We will use the following standard degree estimates of Lachaud and Rolland \cite{lachaud2015number}.
They define $\cdeg(Y)$ to be the sum of the degrees of the irreducible components of an algebraic set $Y$.

\begin{lemma}\cite[Theorem 2.1]{lachaud2015number}\label{lem:affine-degree-estimate}
\begin{enumerate}
\item Let $Y\subseteq \A^N$ be an affine algebraic set over $\Fq$ of dimension $k$ . Then
\[
|Y(\Fq)|\leq \cdeg(Y)q^k.
\]
\item Let $Y=Z(f_1, \dots,f_r) \subseteq \A^N$ be a complete intersection. Then
\[
|Y(\Fq)|\leq 
\Big(\prod_{j=1}^r\deg(f_j)\Big)
q^{N-r}.
\]
\item Let $Y\subseteq \PP^N$ be a projective algebraic set over $\Fq$ of dimension $k$ . Then
\[
|Y(\Fq)|\leq \cdeg(Y)\pi_k(q).
\]
\item Let $Y=V(F_1, \dots,F_r) \subseteq \PP^N$ be a complete intersection. Then
\[
|Y(\Fq)|\leq 
\Big(\prod_{j=1}^r\deg(F_j)\Big)
\pi_{N-r}(q).
\]
\end{enumerate}
\end{lemma}
\begin{definition}[Multiplicity and residual scheme]\label{def: residual scheme}
Let $K$ be a field and let $X\subseteq \PP^m_K$ be a projective scheme.

\begin{enumerate}
    \item Let $C$ be an irreducible component of $X$, and let $\eta_C$ be its generic point.
    The \emph{multiplicity} of $C$ in $X$ is
    \[
    \mult_C(X)
    :=
    \length_{\mathcal O_{X,\eta_C}}
    \mathcal O_{X,\eta_C}.
    \]

    \item Suppose that $X$ is a complete intersection and $X'\subseteq X$ is a closed subscheme.
    The \emph{residual subscheme} of $X'$ in $X$ is the subscheme $X''$ with homogeneous ideal
    \[
    I(X'')
    =
    \operatorname{Ann}_{K[X_0,\dots,X_m]}
    \bigl(I(X')/I(X)\bigr).
    \]
\end{enumerate}
\end{definition}

The following proposition reduces us to the case where $t_W=0$.

\begin{proposition}\label{prop: projective conjecture in (m-1,k) implies tW=1 case}
Fix $0 \leq k \leq m-2$. Let $W \subseteq S_d(m,\Fq)$ such that 
\begin{itemize}
    \item $\dim_{\Fq}(W)=m-k$,
    \item $\dim(V(W)) \leq k$.
\end{itemize}
If we know that:
\begin{enumerate}
    \item $t_W=1$,
    \item $e^{\PP}_{m-1-k}(d,m-1,k;q) \leq f^{\PP}_{m-1-k}(d,m-1,k;q)$.
\end{enumerate}
Then we have
\[
|V(W)(\Fq)| 
\leq f^{\PP}_{m-k}(d,m,k;q)
=d^{m-k}q^k+\pi_{k-1}(q).
\]
\end{proposition}
\begin{proof}
Since $t_W=1$, we can choose a basis of $W$ such that $W=\langle LG_1,F_2,\dots, F_{m-k}\rangle$ where $L \in S_1(m)$ and $G_1$ is of degree $d-1$. Then,
\[
V(W)=
V(L,F_2, \dots, F_{m-k}) \cup 
\big( V(G_1, F_2, \dots, F_{m-k}) \backslash V(L)\big).
\]
We know that $V(L,F_2, \dots, F_{m-k})=V(\Bar{F_2}, \dots, \bar{F}_{m-k})$ where 
\[
\Bar{F_i} \in 
S_d(m,\Fq)/LS_{d-1}(m,\Fq) 
\cong S_{d}(m-1,\Fq).
\] 
Since $t_W=1$, these are $m-k-1$ linearly independent polynomials. Moreover,\\ 
$\dim(V(\bar{F_2}, \dots, \bar{F}_{m-k})) \leq k$. Hence, by assumption (2), we have
\[
|V(L,F_2, \dots, F_{m-k})(\Fq) |
\leq d^{m-k-1}q^{k}+\pi_{k-1}(q).
\]

If $d=1$, then $G$ is a constant and $V(G,F_2, \dots, F_{m-k})=\emptyset$, so the desired bound follows immediately. Otherwise, $G_1, \dots, F_{m-k}$ form a regular sequence. Let $g_1,f_2, \dots, f_{m-k}$ be the dehomogenization with respect to $L$, we have that 
\[
V(G_1, F_2, \dots, F_{m-k}) \backslash V(L)
=Z(g_1, f_2, \dots, f_{m-k}) 
\subseteq \A^{m}.
\] 
By Lemma \ref{lem:affine-degree-estimate}, we have
\[
|\big(V(G_1, F_2, \dots, F_{m-k}) \backslash V(L)\big) (\Fq)|
\leq (d-1)d^{m-k-1}q^k.
\] 
This completes the proof.
\end{proof}

\begin{corollary}\label{cor: tw eq 1 case}
When $k=m-2$ and $t_W=1$, then we have $|V(W)(\Fq)| \leq d^2q^{m-2}+\pi_{m-3}(q)$.
\end{corollary}
\begin{proof}
In this case assumption (2) becomes $e_1^{\PP}(d,m-1;q) = dq^{m-2}+\pi_{m-3}(q)$, which was proven by Serre \cite{serre1991lettre}.
\end{proof}

In the proof of Theorem \ref{prop: r=m-k,k=m-2}, we are left with the case $t_W=0$. 
Hence, in the rest of this section, we assume that no nonzero element of $W$ is divisible by an $\Fq$-linear form.

Throughout the rest of the section, for an algebraic set $X$, by an \emph{$\Fq$-irreducible component} of $X$ we mean an irreducible component of $X$ viewed as a scheme over $\Fq$; its \emph{degree} is the sum of the degrees of the (finitely many, Galois-conjugate) geometric components lying over it, and its \emph{multiplicity} in $X$ is the length of the local ring of $X$ at its generic point, as in Definition \ref{def: residual scheme}. In particular, any $\Fq$-rational linear component is geometrically irreducible and its multiplicity remains the same after extending the base field.

\subsection{If $V(W)$ contains an $\Fq$-rational linear component of multiplicity $1$}

In this section we consider the case when $V(W)$ has a $\Fq$-rational linear component with multiplicity $1$ (and $t_W=0$).
We will show the following proposition.

\begin{proposition}\label{prop: case 1 mult 1 linear component}
Let $W=\langle F_1,F_2\rangle$, where $(F_1,F_2)$ is a regular sequence in $S(m,\Fq)$. 
Let $X=V(W) \subseteq \PP^m$. 
Suppose
\begin{enumerate}
\item $X$ has a $\Fq$-rational linear component $l$ appearing with multiplicity $1$;
\item $t_W=0$;
\item $m \geq 3$.
\end{enumerate}
Then we have:
\begin{equation}\label{eq: multiplicity 1 linear component}
|X(\Fq)| \leq d^2q^{m-2}+\pi_{m-3}(q)-(d^2-1)q^{m-3}.
\end{equation}
\end{proposition}

Write $X=l \cup Y$ where $Y$ is the residual scheme of $l$ in $X$. Then $|X(\Fq)|=|l(\Fq)|+|(Y \backslash l) (\Fq)|$. After linear change of coordinates over $\Fq$, we may assume $l=V(X_0,X_1)$. Then we can write the ideal of $Y$ explicitly using the following formula of linkage ideal:
\begin{lemma}\cite[Theorem 21.23]{EisenbudCA}\label{lem: residual ideal}
 Write
\[
F_1=X_0A+X_1B,\qquad F_2=X_0C+X_1D,
\]
where $A,B,C,D \in S_{d-1}(m,\Fq)$. Let $\Delta=AD-BC$. Then $Y=V(F_1,F_2,\Delta)$.
\end{lemma}

We write $Y|_l$ for the scheme-theoretic intersection $Y\cap l$, viewed as a subscheme of $l$.
\begin{lemma}\label{lem: degree of Y on l}
$Y|_l$ is a hypersurface in $l$ of degree $2d-2$.
\end{lemma}
\begin{proof}
$Y|_l = V(F_1,F_2, \Delta)|_l=V(\Delta)|_l$.
Since \(l\) appears with multiplicity one in \(X\), the restriction
$\Delta|_{l}\in \Fq[X_2,\ldots,X_m]_{2d-2}$ is a nonzero degree $2d-2$ polynomial. Hence, it is a hypersurface of degree $2d-2$.
\end{proof}
Now, we count $|(X\backslash l)(\Fq)|$. For each point $P \in (X\backslash l)(\Fq)$, there is a unique $\Fq$-rational hyperplane containing $l$ and $P$. Hence, 
let \(\check{l}(\Fq)\) denote the set of \(\Fq\)-hyperplanes containing
\(l\). We have that
\[
|(X\backslash l)(\Fq)|
=\sum_{H \in \lc(\Fq)}
|\big((X \cap H) \backslash l\big)(\Fq)|.
\]
We also know that $|\check{l}(\Fq)|=q+1$.
Fix \(H\in \check{l}(\Fq)\). After linear change of variables we can assume that $H=V(X_0)$. Inside \(H\simeq \PP^{m-1}\), $l=V(X_1)$. Let $\Fb_i$ denote the image of $F_i$ in $S(m,\Fq)/(X_0) \cong S(m-1,\Fq)$. Let $X_1K_H=\gcd(\Fb_1, \Fb_2)$, then we can write
\[
\Fb_1=X_1 K_H \Fb_1',
\qquad
\Fb_2=X_1 K_H \Fb_2',
\qquad
a_H=\deg(K_H)
\]
where $\gcd(\Fb_1',\Fb_2')=1$. Let $f_1, f_2, k_H$ be the dehomogenization of $\Fb_1',\Fb_2',K_H$ with respect to $X_1$. 
Then, $(X\cap H) \backslash l \cong (Z(k_H)\bigcup Z(f_1,f_2))\subseteq \A^{m-1}$. Hence, by Lemma \ref{lem:affine-degree-estimate}, we have
\begin{equation}\label{Eqn: in terms of aH}
|((X \cap H) \backslash l)(\Fq)| 
\leq a_H q^{m-2}+(d-1-a_H)^2q^{m-3}
=q^{m-3}\big[(d-1)^2+a_H(q-2d+2+a_H)\big].
\end{equation}
Therefore, to bound $|((X \cap H) \backslash l) (\Fq)|$, we want to bound $a_H$ and $\sum_{H \in \lc(\Fq)}a_H$.

\begin{lemma}\label{lem: bound on the sum}
We have
\begin{enumerate}
\item $a_H \leq d-2$;
\item $\sum_{H \in \lc(\Fq)}a_H \leq 2d-2$.
\end{enumerate}
\end{lemma}

\begin{proof}
\begin{enumerate}
\item If $a_H=d-1$, then
$\Fb_1', \Fb_2'$ are constants. Hence, there exist $\alpha, \beta \in \Fq$, not both zero such that
$\alpha \Fb_1+\beta \Fb_2=0$. 
Since $H=V(X_0)$, there is $G \in S_{d-1}(m,\Fq)$ such that $X_0G=\alpha F_1+\beta F_2 \in W$. This violates the assumption that $t_W=0$. Hence, we must have $a_H \leq d-2$.
\item
We write $Y=\left(\bigcup_{H \in \lc(\Fq)}Y_H \right)\bigcup Z$, where $Y_H$ is union of $\Fq$-irreducible components contained in $H$ and $Z$ is the union of $\Fq$-irreducible components not contained in any $H \in \lc(\Fq)$.
By Lemma \ref{lem: degree of Y on l}, we have that $\deg(Y|_l)=2d-2$. 

On the other hand, we claim that $a_H=\deg(Y_H)=\deg({Y_H}|_l)$. Indeed, 
\[Y \cap H=(Y \cap H)_{m-2} \cup (Y \cap H)_{ m-3}\]
Here $(Y \cap H)_k$ denotes the union of dimension $k$ components. $(Y \cap H)_{m-2}=V(K_H)$ is pure of dimension $m-2$.
Hence, $a_H=\deg(K_H)=\deg(Y_H)$. Moreover, since $l$ has multiplicity one in $X$, we know that $K_H$ does not vanish on the entire $l$. 
So $V(K_H)$ intersects $l$ properly and hence $\deg(V(K_H)|_l)=\deg(V(K_H))=a_H$.  Therefore,
\[
\sum_{H\in \check{l}(\Fq)}a_H
=\sum_{H \in \lc(\Fq)}\deg({Y_H}|_l) \leq 
\deg(Y|_l)
=
2d-2. \qedhere
\]
\end{enumerate}
\end{proof}
\begin{proof}[Proof of Proposition \ref{prop: case 1 mult 1 linear component}]
By \eqref{Eqn: in terms of aH} and Lemma~\ref{lem: bound on the sum}, we get
\begin{equation*}
\begin{split}
|X(\Fq)|
&=|l(\Fq)|+|(X\backslash l)(\Fq)| \\
&\leq \pi_{m-2}(q)
+\sum_{H \in \lc(\Fq)}
q^{m-3}\big[(d-1)^2+a_H(q-d)\big]\\
& \leq \pi_{m-2}(q)
+(d-1)^2(q+1)q^{m-3}
+q^{m-3}(q-d)\sum_{H \in \lc(\Fq)}a_H\\
&\leq \pi_{m-3}(q)
+q^{m-3}\big[q+(d-1)^2(q+1)+(2d-2)(q-d)\big]\\
&=d^2q^{m-2}
+\pi_{m-3}(q)-(d^2-1)q^{m-3}.
\end{split}
\end{equation*}
This finishes the proof of Proposition \ref{prop: case 1 mult 1 linear component}.
\end{proof}
\begin{proposition}\label{prop: case 1 mult 1 and 2 linear component}
Let $W=\langle F_1,F_2\rangle$ where $(F_1,F_2)$ is a regular sequence in $S(m,\Fq)$. Let $X=V(W) \subseteq \PP^m$. Suppose
\begin{enumerate}
    \item $X$ has a $\Fq$-rational linear component $l$ with multiplicity $1$;
    \item $X$ has a $\Fq$-rational linear component $l'$ with multiplicity $ \geq 2$;
    \item $m \geq 3$.
\end{enumerate}
Then we have:
\begin{equation}\label{eq: multiplicity 1 and 2 linear component}
|X(\Fq)| \leq (d^2-1)q^{m-2}+\pi_{m-3}(q)-(d-1)^2q^{m-3}.
\end{equation}
\end{proposition}
\begin{proof}
Similar to the proof of Proposition \ref{prop: case 1 mult 1 linear component}, we write $X=l \cup Y$ where $Y$ is the residual scheme of $l$ in $X$, and we write $|X(\Fq)|=l(\Fq)+\sum_{H \in \lc(\Fq)}|\big((X \cap H )\backslash l\big)(\Fq)|$. We write 
\[\Fb_1=X_1 K_H \Fb_1',\qquad
\Fb_2=X_1 K_H \Fb_2', \qquad a_H=\deg(K_H).\] 
We have $(X\cap H) \backslash l \cong \big(Z(k_H)\bigcup Z(f_1,f_2)\big)\subseteq \A^{m-1}$ as before.

Consider the conclusion of Lemma \ref{lem: bound on the sum}. Since we removed the assumption that $t_W=0$, we only have $a_H \leq d-1$. The proof of Lemma \ref{lem: bound on the sum} of this part goes through and gives us that $\sum_{H \in \lc(\Fq)}a_H \leq 2d-2$.  We consider $|((X \cap H) \backslash l)(\Fq)|$ depending on whether $H$ contains $l'$ or not.

\begin{itemize}
    \item If $H$ does not contain $l'$: then, $H$ intersect $l'$ properly. Hence, \[\text{mult}_{l' \cap H}(X \cap H) \geq\text{mult}_{l'}(X) \geq 2\] 
    In particular, $\deg(V(\Fb_1', \Fb_2')^{\mathrm{red}}) \leq \deg(V(\Fb_1', \Fb_2'))-1$. Hence, \\
    $|Z(f_1,f_2)(\Fq)| \leq ((d-1-a_{H})^2-1)q^{m-3}$. We get that
    \[
    |\big((X \cap H) \backslash l\big)(\Fq)| 
    \leq q^{m-3}\big[a_Hq+(d-1-a_{H})^2-1\big] 
    \leq q^{m-3}\left[(d^2-2d)+a_H(q-d+1)\right].
    \]
    Here we used $a_H \leq d-1$.
    \item If $H$ contains $l'$: then, we have 
    \[|((X \cap H) \backslash l)(\Fq)| \leq q^{m-3}\big[a_Hq+(d-1-a_{H})^2\big] \leq q^{m-3}\left[(d^2-2d+1)+a_H(q-d+1)\right]\]
\end{itemize}
Since $l$ and $l'$ are different linear component, there is at most $1$ hyperplane in $\PPc^m(\Fq)$ containing both of them. Hence,
\begin{equation*}
    \begin{split}
        |X(\Fq)| &= |l(\Fq)|+\sum_{H \in \lc(\Fq)}|\big((X \cap H )\backslash l\big)(\Fq)| \\
        & \leq \pi_{m-2}(q)+\sum_{H \in \lc(\Fq), H \not \in \lc'(\Fq)}\big[(d^2-2d)+a_H(q-d+1)\big]q^{m-3}\\
        &+ \sum_{H \in \lc(\Fq) \cap \lc'(\Fq)}\big[(d^2-2d+1)+a_H(q-d+1)\big]q^{m-3}\\
        & \leq \pi_{m-2}(q)+q^{m-3}\big[(q+1)(d^2-2d)+1+\sum_{H \in \lc(\Fq)}a_H(q-d+1)\big] \\
        & \leq \pi_{m-2}(q)+q^{m-3}\big[(2d-2)(q-d+1)+(q+1)(d^2-2d)+1\big]\\
        & =\pi_{m-3}(q)+q^{m-3}\big[(d^2-1)q-(d-1)^2\big]\\
        &=\pi_{m-3}(q)+(d^2-1)q^{m-2}-(d-1)^2q^{m-3}.
    \end{split}
\end{equation*}
This proves \eqref{eq: multiplicity 1 and 2 linear component}.
\end{proof}

\subsection{If $V(W)$ contains a $\Fq$-rational linear component with multiplicity $\geq 2$}

\begin{proposition}\label{prop: case 2: multi 2 linear component}
Let $W=\langle F_1,F_2\rangle$, where $(F_1,F_2)$ is a regular sequence in $S_d(m,\Fq)$. Let $X=V(W) \subseteq \PP^m$. Suppose
\begin{enumerate}
    \item There is a $\Fq$-rational linear component of $X$ with multiplicity $ \geq 2$;
    \item $m \geq 3$.
\end{enumerate}
Then we have:
\begin{equation}\label{eq: inequality linear 2}
|X(\Fq)| \leq (d^2-1)q^{m-2}+\pi_{m-3}(q).
\end{equation}
\end{proposition}

\begin{remark}
Here, we remove the assumption that $t_W=0$. Hence, if $X$ has a linear component of multiplicity $\geq 2$, then even if $t_W=1$, we are bounded away from the target inequality. 
The reason for removing this assumption is that even if $X$ has $t_W=0$, this might not be true when restricted to a hyperplane. 
To apply the inductive argument, we can not have $t_W=0$ in the statement. 
\end{remark}

\begin{proof}
We induct on $m$, the dimension of the ambient projective space. 

Let $l_1,\ldots,l_t$ be the distinct $\Fq$-rational linear components of $X$, at least one of which has multiplicity at least $2$. Let $\Set$ be the set of $\Fq$-rational hyperplanes that contain some non-linear $\Fq$-component of $X$. For each $H \in \Set$, let $\Gamma_H$ be the union of the non-linear $\Fq$-irreducible components of $X$ contained in $H$. Let $Z$ be the union of the components not contained in any $\Fq$-rational hyperplane. Then we can write
\[
X
=
\left(\bigcup_{i=1}^t l_i\right)
\cup
\left(\bigcup_{H\in \mathcal S}\Gamma_H\right)
\cup Z.
\]
 Put
\[
s=|\mathcal S|,
\qquad
B=\sum_{H\in \mathcal S}\deg(\Gamma_H), \qquad D=\deg(Z)
\]
If $X$ has some linear component with multiplicity $1$, then the conclusion follows from Proposition~\ref{prop: case 1 mult 1 and 2 linear component}. Therefore, we can assume that every linear component of $X$ has multiplicity $\geq 2$. Hence, 
we know that $B \geq 2s$ and $B+2t+D \leq d^2$.

When $m=2$, $X$ is a zero dimensional scheme of degree $d^2$. Moreover, there exists a $\Fq$-rational point of $X$ with multiplicity $\geq 2$. Hence, $|X(\Fq)| \leq d^2-1=(d^2-1)q^{m-2}+\pi_{m-3}(q)$. This proves the base case for our induction.


Now, we consider the general $m \geq 3$. On each $\Gamma_H$, we use the fact that it is a hypersurface in $H$ of degree $\deg(\Gamma_H)$ with no $\Fq$-rational linear component. By Lemma \ref{lem: no linear divisor hypersurface},
\begin{align*}
\begin{split}
|\Gamma_H(\Fq)| 
&\leq (\deg(\Gamma_H)-1)q^{m-2}+\deg(\Gamma_H)q^{m-3}+\pi_{m-4}(q),\\
|\left(\bigcup_{H \in \Set}\Gamma_H\right)(\Fq)|
& \leq (B-s)q^{m-2}+Bq^{m-3}+s\pi_{m-4}(q).\\ 
\end{split}
\end{align*}
And by Lemma \ref{lem:affine-degree-estimate} we get:
\[
\left|\left(\bigcup_{i=1}^t l_i\cup Z\right)(\Fq)\right|
\leq (D+t)\pi_{m-2}(q)
\leq (d^2-B-t) \pi_{m-2}(q).\]
Hence,
\[
|X(\Fq)| 
\leq \big[d^2-(t+s)\big]q^{m-2}
+(d^2-t)q^{m-3}
+(d^2-B-t+s)\pi_{m-4}(q).
\]
We separate into two cases based on $t+s$.
\begin{enumerate}
\item If $t+s \geq d+1$. Then,
\begin{equation*}
    \begin{split}
        |X(\Fq)| & 
        \leq (d^2-d-1)q^{m-2}+(d^2-t)q^{m-3}+(d^2-s-t)\pi_{m-4}(q)\\
        & \leq (d^2-d-1)q^{m-2}+(d^2-1)q^{m-3}+(d^2-d-1)\pi_{m-4}(q).\\
    \end{split}
\end{equation*}
    To show that $|X(\Fq)| \leq (d^2-1)q^{m-2}+\pi_{m-3}(q)$, it suffices to show that the difference 
    \begin{equation}\label{eq: diff geq 0}
    dq^{m-2}-(d^2-2)q^{m-3}-(d^2-d-2)\pi_{m-4}(q) \geq 0.
    \end{equation}
    Since $q \geq d+1$, we have
    \begin{equation*}
        \begin{split}
            (d-1)q^{m-2} & \geq (d^2-2)q^{m-3}\\
            q^{m-2} & \geq (d^2-d-2)\frac{q^{m-3}}{q-1} \geq (d^2-d-2)\pi_{m-4}(q).\\
        \end{split}
    \end{equation*}
    Hence we get \eqref{eq: diff geq 0} as desired and hence prove \eqref{eq: inequality linear 2} in the case $t+s \geq d+1$.
    \item If $t+s \leq d$, then we separate further into two subcases.
    \begin{enumerate}
    \item If $X(\Fq) \subseteq \left(\bigcup_{i=1}^t l_i  \right) \cup  \left( \bigcup_{H \in \Set} H \right)$. 
    \begin{itemize}
        \item If $s=0$. Then, $X(\Fq) \subseteq \left(\bigcup_{i=1}^t l_i\right)$. So 
        \[
        |X(\Fq)|
        \leq t\pi_{m-2}(q) 
        \leq d\pi_{m-2}(q) 
        \leq (d^2-1)q^{m-2}+\pi_{m-3}(q).
        \]
        \item If $s \geq 1$. Then fix an arbitrary $H_1 \in \Set$. We have
        \[
        |X(\Fq)| 
        \leq |(X \cap H_1)(\Fq)|
        +\sum_{H' \in \Set, H' \neq H_1} |\big( (X \cap H') \backslash H_1\big)(\Fq)|
        + \sum_{i=1}^t |(l_i \backslash H_1)(\Fq)|.
        \]

        By the known case of Conjecture \ref{Conj: complete GDC} for $r =1$, we know that 
        \begin{align*}
            \begin{split}
            |(X \cap H_1)(\Fq)| &\leq dq^{m-2}+\pi_{m-3}(q), \\
            |\big( (X \cap H') \backslash H_1\big)(\Fq)| & \leq dq^{m-2}.
            \end{split}
        \end{align*}
        Hence, 
        \[
        |X(\Fq)| 
        \leq sdq^{m-2}+\pi_{m-3}(q)+tq^{m-2} 
        = (ds+t)q^{m-2}+\pi_{m-3}(q).
        \] 

        We want to show that $ds+t \leq d^2-1$. Since $1 \leq s \leq d-t \leq d-1$, we have $ds+t \leq ds+d-s \leq d^2-1$. Hence, \eqref{eq: inequality linear 2} is true in this case. 
    \end{itemize}
    
    \item If there exists $P \in \left(X \backslash(\bigcup_{i=1}^t l_i \cup  \bigcup_{H \in \Set} H)\right)(\Fq)$:
       Then, we fix such a point $P$ and let $\Pc(\Fq) \subseteq \PPc^m(\Fq)$ be the set of $\Fq$-rational hyperplanes passing through $P$. Consider the incidence set
\[
T_P
=
\{(x,H): x\in X(\Fq)\setminus\{P\},H \in \Pc(\Fq), x \in H\}.
\]
For each \(x\in X(\Fq)\setminus\{P\}\), the number of \(\Fq\)-hyperplanes containing
both \(P\) and \(x\) is \(\pi_{m-2}(q)\). Hence
\[
|T_P|
=
\left(|X(\Fq)|-1\right)\pi_{m-2}(q).
\]
On the other hand, the fiber over each $H' \in \Pc(\Fq)$ is $(X \cap H')(\Fq)\backslash \{P\}$. If $H'$ is an $\Fq$-hyperplane containing $P$, then $H'$ must intersect each $\Gamma_{H}$ properly for every $H \in \Set$. Moreover, for each linear component $l_i$, there is at most one hyperplane $H' \in \Pc(\Fq)$ that contains $l_i$. So, there are at most $t$ hyperplanes $H' \in \Pc(\Fq)$ that contain some linear component $l_i$.

Let $\mathcal{T}=\{H' \in \Pc(\Fq) \mid H' \text{ contains some linear component $l_i$}\}$. For $H' \in \Pc(\Fq)$, we consider $|(X \cap H')(\Fq)|$.
\begin{itemize}
    \item If $H' \in \TT$. Let $b_{H'}$ be the number of distinct linear components contained in $H'$. Then,
    \[
    |(X \cap H')(\Fq) \backslash \{P\} | 
    \leq b_{H'}q^{m-2}
    +\pi_{m-3}(q)
    +(d-b_{H'})^2q^{m-3}-1.
    \]
    \item If $H' \not \in \TT$. Then, $H'$ also intersects all the linear components $l_i$ properly. Let $h'$ be the linear form such that $H'=V(h')$. 
    Then, $X \cap H'=V(\Fb_1,\Fb_2)$ is a complete intersection in $H' \cong \PP^{m-1}$. Moreover, 
    \[
    \text{mult}_{l_i \cap H'}(X \cap H') \geq \text{mult}_{l_i}(X).
    \] 
    Therefore, $X \cap H'$ has at least one linear component with multiplicity at least 2. Hence it satisfies the assumptions stated in Proposition \ref{prop: case 2: multi 2 linear component}. 
    Applying the inductive assumption to $|X \cap H'|$, we have
    \[
    |(X \cap H')(\Fq) \backslash \{P\} | \leq (d^2-1)q^{m-3}+\pi_{m-4}(q)-1.
    \]
    \end{itemize}

    Hence, we have
\begin{equation*}
\begin{split}
\pi_{m-2}(q)\big(|X(\Fq)|-1\big)
&=\sum_{H'\in\Pc(\Fq)}
|(X\cap H')(\Fq)\backslash\{P\}|\\
&\leq
\sum_{H'\in\TT}
\left[
b_{H'}q^{m-2}+\pi_{m-3}(q)+(d-b_{H'})^2q^{m-3}-1
\right]\\
&\quad+
\sum_{H'\in\Pc(\Fq)\backslash\TT}
\big(
(d^2-1)q^{m-3}+\pi_{m-4}(q)-1
\big)\\
&=
\pi_{m-1}(q)
\big(
(d^2-1)q^{m-3}+\pi_{m-4}(q)-1
\big)\\
&\quad+
q^{m-3}\sum_{H'\in\TT}
\big(
b_{H'}(q-2d+b_{H'})+2
\big).
\end{split}
\end{equation*}
Since $P\notin l_i$, for each linear component $l_i$ there is a unique hyperplane through $P$ containing $l_i$. Hence, $\sum_{H'\in\TT}b_{H'}=t.$
Therefore $\sum_{H'\in\TT}b_{H'}^2\leq t^2$. Also, we have $|\TT|\leq t$. Therefore,
\[
\begin{split}
\sum_{H'\in\TT}
\big[ b_{H'}(q-2d+b_{H'})+2 \big]
&\leq t(q-2d)+t^2+2t\\
&=t(q-2d+t+2)\leq d(q-d+2),
\end{split}
\]
where the last inequality follows from $t\leq d$.
Thus,
\[
\begin{split}
\pi_{m-2}(q)\big(|X(\Fq)|-1\big)
&\leq
\pi_{m-1}(q)
\big(
(d^2-1)q^{m-3}+\pi_{m-4}(q)-1
\big)\\
&\quad+d(q-d+2)q^{m-3}.
\end{split}
\]
To show \eqref{eq: inequality linear 2}, it suffices to show that the difference
\begin{equation}\label{eq: difference geq 0 3}
q\pi_{m-2}(q)
-\big[(d^2-1)q^{m-3}+\pi_{m-4}(q)-1+d(q-d+2)q^{m-3}\big]
\geq 0.
\end{equation}
Indeed,
\begin{align*}
& q\pi_{m-2}(q)
-\big[(d^2-1)q^{m-3}+\pi_{m-4}(q)-1+d(q-d+2)q^{m-3}\big]\\
&=
\big(q^2+q-dq-2d+1\big)q^{m-3}
+(q-1)\pi_{m-4}(q)+1\\
&=
\big[q(q-d+1)-2d+1\big]q^{m-3}
+(q-1)\pi_{m-4}(q)+1.
\end{align*}
Since $q\geq d+1$, we have
\[
q(q-d+1)-2d+1
\geq 2q-2d+1>0.
\]
Hence, \eqref{eq: difference geq 0 3} holds, which proves \eqref{eq: inequality linear 2} in this case.
\end{enumerate}
\end{enumerate}
This finishes the proof of \eqref{eq: inequality linear 2} and hence Proposition \ref{prop: case 2: multi 2 linear component}.
\end{proof}

\subsection{If $V(W)$ has no $\Fq$-rational linear component}
\begin{proposition}\label{prop: case 3 no linear component}
    Let $W=\langle F_1,F_2\rangle$, where $(F_1,F_2)$ is a regular sequence in $S(m,\Fq)$. Let $X=V(W) \subseteq \PP^m$. Suppose 
    \begin{itemize}
    \item $t_W=0$;
    \item $m \geq 3$;
    \item $X$ contains no $\Fq$-rational linear component;
    \item $e^{\PP}_{2}(d,m-1,m-3;q) \leq f^{\PP}_{2}(d,m-1,m-3;q)$.
    \end{itemize}
    Then we have
    \begin{equation}\label{eq: non-liner component bound}
    |X(\Fq)| \leq d^2q^{m-2}+\pi_{m-3}(q)-(q-d+2)
    \end{equation}
\end{proposition}
\begin{proof}

We use the notation $\Set$, $\Gamma_H$, $Z$, $s$, $B$, and $D$ from the proof of Proposition \ref{prop: case 2: multi 2 linear component}. Thus, $\Set$ is the set of $\Fq$-rational hyperplanes containing an $\Fq$-irreducible component of $X$, $\Gamma_H$ is the union of the components of $X$ contained in $H$, and $Z$ is the union of the components not contained in any $\Fq$-rational hyperplane. As before, set
\[
s=|\Set|,\qquad
B=\sum_{H\in\Set}\deg(\Gamma_H),\qquad
D=\deg(Z).
\]
Since $X$ has no $\Fq$-rational linear component, we have
\[
X=\left(\bigcup_{H\in\Set}\Gamma_H\right)\cup Z,
\qquad
B+D\leq d^2, \qquad B \geq 2s
\]
\begin{enumerate}
\item If $s \geq d+1$. In this case we show that 
\begin{equation}\label{eq: case 1 target}
|X(\Fq)| \leq (d^2-1)q^{m-2}+\pi_{m-3}(q)
\end{equation}
Notice that this implies \eqref{eq: non-liner component bound} since $q^{m-2} \geq (q-d+2)$.

Let $a_H=\deg(\Gamma_H)$.
Since \(X\) has no $\Fq$-rational linear components, \(\Gamma_H\) viewed as a hypersurface in
\(H\simeq \PP^{m-1}\), has no $\Fq$-rational linear component. Hence, by Lemma \ref{lem: no linear divisor hypersurface}, we have
\[
|\Gamma_H(\Fq)|
\leq
(a_H-1)q^{m-2}+a_Hq^{m-3}+\pi_{m-4}(q).
\]
On the other hand, by Lemma \ref{lem:affine-degree-estimate}, since $Z$ is pure of dimension $m-2$, we have
\[
|Z(\Fq)|\leq \deg(Z)\pi_{m-2}(q).
\]
We obtain
\begin{equation}\label{eq: bound 1}
\begin{split}
|X(\Fq)|
&\leq
\sum_{H\in \mathcal S}|\Gamma_H(\Fq)|+|Z(\Fq)| \\
&\leq
(B-s)q^{m-2}+Bq^{m-3}+s\pi_{m-4}(q)+D\pi_{m-2}(q) \\
&=
(B+D-s)q^{m-2}+(B+D)q^{m-3}+(s+D)\pi_{m-4}(q) \\
&\leq
(d^2-s)q^{m-2}+d^2q^{m-3}+(s+D)\pi_{m-4}(q).
\end{split}
\end{equation}

Since $B\geq 2s$ and $B+D\leq d^2$, we have $D \leq d^2-2s$, and hence $s+D \leq d^2-s$. Then \eqref{eq: bound 1} gives
\[
|X(\Fq)|
\leq
(d^2-d-1)q^{m-2}
+
d^2q^{m-3}
+
(d^2-d-1)\pi_{m-4}(q).
\]
To show \eqref{eq: case 1 target}, it suffices to show that the difference
\[(d^2-1)q^{m-2}+\pi_{m-3}(q)-\big[(d^2-d-1)q^{m-2}
+
d^2q^{m-3}
+
(d^2-d-1)\pi_{m-4}(q)\big] \geq 0\]
We carry out the computation
\begin{equation*}
    \begin{split}
        &(d^2-1)q^{m-2}+\pi_{m-3}(q)-\big[(d^2-d-1)q^{m-2}
+
d^2q^{m-3}
+
(d^2-d-1)\pi_{m-4}(q)\big]\\
&=dq^{m-2}-(d^2-1)q^{m-3}-(d^2-d-2)\pi_{m-4}(q) \\
& \geq q^{m-3}\big[dq-(d^2-1)-\frac{d^2-d-2}{q-1}\big]\geq q^{m-3}\big[dq-(d^2-1)-(d+1)\big] \\
&\geq q^{m-3}\big[dq-d(d+1)\big] \geq 0.
    \end{split}
\end{equation*}
This proves \eqref{eq: case 1 target}.
\item If $s \leq d$. Then we further divide into the following $2$ sub-cases:
\begin{enumerate}
    \item Suppose there exists $P \in \left( X\backslash (\bigcup_{H \in \Set}H)\right)(\Fq)$. 
    
    Then, we fix such a point $P$ and let $\Pc(\Fq) \subseteq \PPc(\Fq)$ be the set of $\Fq$-rational hyperplanes passing through $P$. Consider the incidence set
\[
T_P
=
\{(x,H): x\in X(\Fq)\setminus\{P\},H \in \Pc(\Fq), x \in H\}.
\]
For each \(x\in X(\Fq)\setminus\{P\}\), the number of \(\Fq\)-hyperplanes containing
both \(P\) and \(x\) is \(\pi_{m-2}(q)\). Hence
\[
|T_P|
=
\left(|X(\Fq)|-1\right)\pi_{m-2}(q).
\]
On the other hand, if \(H\) is an \(\Fq\)-hyperplane containing \(P\), then \(H\)
contains no irreducible component of \(X\) by the choice of \(P\). Therefore
\(X\cap H\) is a complete intersection of degree $(d,d)$ in
\(H\simeq \PP^{m-1}\). By the inductive assumption $e^{\PP}_{2}(d,m-1,m-3;q) \leq f^{\PP}_{2}(d,m-1,m-3;q)$, we have
 $|(X\cap H)(\Fq)|
\leq
d^2q^{m-3}+\pi_{m-4}(q)$.
Thus
\[
|T_P|
\leq
\pi_{m-1}(q)\left(d^2q^{m-3}+\pi_{m-4}(q)-1\right).
\]
Therefore
\[
\pi_{m-2}(q)\left(|X(\Fq)|-1\right)
\leq
\pi_{m-1}(q)\left(d^2q^{m-3}+\pi_{m-4}(q)-1\right).
\]
Using $
\pi_{m-1}(q)=q\pi_{m-2}(q)+1$,
we get
\[
\begin{aligned}
|X(\Fq)|
&\leq
1+
\tfrac{\pi_{m-1}(q)}{\pi_{m-2}(q)}
\left(d^2q^{m-3}+\pi_{m-4}(q)-1\right) \\
&=
1+
q\left(d^2q^{m-3}+\pi_{m-4}(q)-1\right)
+
\left\lfloor\tfrac{d^2q^{m-3}+\pi_{m-4}(q)-1}{\pi_{m-2}(q)}\right\rfloor \\
&=
d^2q^{m-2}+\pi_{m-3}(q)-q
+
\left\lfloor \tfrac{d^2q^{m-3}+\pi_{m-4}(q)-1}{\pi_{m-2}(q)}\right\rfloor.
\end{aligned}
\]
To show \eqref{eq: non-liner component bound}, it suffices to show that
\[\left\lfloor \tfrac{d^2q^{m-3}+\pi_{m-4}(q)-1}{\pi_{m-2}(q)}\right\rfloor \leq d-2.\]
Equivalently,
\[
d^2q^{m-3}+\pi_{m-4}(q)-1\leq (d-2)\pi_{m-2}(q)+\pi_{m-2}(q)-1.
\]
Taking the difference, we have
\begin{align*}
\begin{split}
&(d-1)\pi_{m-2}(q)
-\left(d^2q^{m-3}+\pi_{m-4}(q)\right) \\
&=(d-1)(q^{m-2}+q^{m-3}+\pi_{m-4}(q))
-d^2q^{m-3}-\pi_{m-4}(q) \\
&=\bigl[(d-1)(q+1)-d^2\bigr]q^{m-3}
+(d-2)\pi_{m-4}(q).
\end{split}
\end{align*}
This is non-negative since $q \geq d+1$. Hence, we have shown \eqref{eq: non-liner component bound}.

\item If such point does not exist, then $X(\Fq)\subseteq \bigcup_{H\in \mathcal S}H$. In this case we show \eqref{eq: case 1 target}. If $s=0$ there is nothing to prove. If $s \geq 1$, fix an arbitrary $H_1 \in \Set$, we have
\[
|X(\Fq)|
\leq |(X \cap H_1)(\Fq)|+\sum_{H \in \Set, H \neq H_1}|\big((X \cap H)\backslash H_1\big)(\Fq)|
\]
For each $H \in \Set$, let $H=V(h)$. We claim that $F_1,F_2$ are linearly independent mod $h$. If not, then $hG=\alpha F_1+\beta F_2$, violating the assumption that $t_W=0$. Hence, by Theorem~\ref{thm: affine BDG} and the known case of Conjecture \ref{Conj: complete GDC} with $r=2$ and $m-1$ variables, we have that
\[
|(X\cap H_1)(\Fq)|
\leq
(d-1)q^{m-2}+q^{m-3}+\pi_{m-3}(q),
\]
\[
|\big( (X\cap H) \backslash H_1\big)(\Fq)|
\leq
(d-1)q^{m-2}+q^{m-3}.
\]
Thus
\begin{align*}
\begin{split}
|X(\Fq)|
&\leq
s(d-1)q^{m-2}+sq^{m-3}+\pi_{m-3}(q)\\
&\leq d(d-1)q^{m-2}+dq^{m-3}+\pi_{m-3}(q) \\
&=(d^2-1)q^{m-2}+\pi_{m-3}(q)-\big[(d-1)q^{m-2}-dq^{m-3}\big].\\
\end{split}
\end{align*}
Since $q \geq d+1$, we have $(d-1)q^{m-2}-dq^{m-3} \geq 0$. Hence, we get \eqref{eq: case 1 target} as desired.  \qedhere
\end{enumerate}
\end{enumerate}
\end{proof}
As a summary, we present the proof of Theorem \ref{prop: r=m-k,k=m-2}.
\begin{proof}[Proof of Theorem \ref{prop: r=m-k,k=m-2}]
Let $W \subseteq S_d(m, \Fq)$ such that $W=\langle F_1,F_2 \rangle$ and $(F_1,F_2)$ is a regular sequence in $S(m,\Fq)$. Let $X=V(W)$. We show that $|X(\Fq)| \leq \fptm$ by inducting on $m$.

If $m=2$, then $\fptm=d^2$. Moreover, by Bezout's theorem $X$ is a zero-dimensional scheme of degree $d^2$ so $|X(\Fq)| \leq \deg(X) \leq d^2$.

Now assume that $m \geq 3$ and that $e^{\PP}_2(d,m-1,m-3;q)=f^{\PP}_2(d,m-1,m-3;q)$. 
\begin{enumerate}
    \item If $t_W=1$: the result follows from Corollary \ref{cor: tw eq 1 case}. Furthermore, if $X$ contains a linear component of multiplicity $\geq 2$, by Proposition \ref{prop: case 2: multi 2 linear component}, we have \[|X(\Fq)| \leq \fptm-q^{m-2}.\]
    \item If $t_W=0$:
    \begin{enumerate}
        \item If $X$ contains no $\Fq$-rational linear component. Then by Proposition \ref{prop: case 3 no linear component}, $|X(\Fq)| \leq \fptm-(q-d+2)$.
        \item If $X$ contains a $\Fq$-rational linear component of multiplicity $1$. Then by Proposition \ref{prop: case 1 mult 1 linear component}, $|X(\Fq)| \leq \fptm-(d^2-1)q^{m-3}$.
        \item If $X$ contains $\Fq$-rational linear component of multiplicity $\geq 2$, by Proposition \ref{prop: case 2: multi 2 linear component}, we have $|X(\Fq)| \leq \fptm-q^{m-2}$.
    \end{enumerate}
\end{enumerate}
This finishes inductive step and hence the proof of Theorem \ref{prop: r=m-k,k=m-2}.
\end{proof}

\section*{Acknowledgments}
We thank Sudhir Ghorpade for introducing us to this problem. We thank Deepesh Singhal and Nathan Kaplan for many helpful discussions about the problem.

\bibliographystyle{plain}
\bibliography{Bibl.bib}
\end{document}